\documentclass{mcom-l}

\usepackage{color}
\usepackage{amsmath}
\usepackage{amsfonts}
\usepackage{amssymb}%
\usepackage{latexsym}
\usepackage[dvips]{graphicx}
\usepackage{psfrag}
\usepackage{multirow}

\newtheorem{theorem}{Theorem}[section]
\newtheorem{corollary}[theorem]{Corollary}
\newtheorem{definition}[theorem]{Definition}
\newtheorem{lemma}[theorem]{Lemma}
\newtheorem{proposition}[theorem]{Proposition}

\newtheorem{assumption}[theorem]{Assumption}
\newtheorem{example}[theorem]{Example}
\newtheorem{remark}[theorem]{Remark}

\numberwithin{equation}{section}

\usepackage[a4paper]{geometry}
\newcommand{\hprob}{h_{\mathrm{prob}}}
\newcommand{\epsprob}{\eps_{\mathrm{prob}}}
\newcommand{\Hprec}{H_{\mathrm{prec}}}
\newcommand{\epsprec}{\eps_{\mathrm{prec}}}
\newcommand{\R}{\mathbb{R}}

\newcommand{\cC}{{\mathcal C}}

\newcommand{\cI}{{\mathcal I}}

\newcommand{\cT}{{\mathcal T}}
\newcommand{\cK}{{\mathcal K}}

\newcommand{\cV}{{\mathcal V}}

\newcommand{\cN}{{\mathcal N}}

\newcommand{\cO}{{\mathcal O}}

\newcommand{\bC}{\mathbf{C}}

\newcommand{\br}{\mathbf{r}}
\newcommand{\bV}{\mathbf{V}}
\newcommand{\bW}{\mathbf{W}}

\newcommand{\bfv}{\mathbf{v}}
\newcommand{\bu}{\mathbf{u}}

\newcommand{\by}{\mathbf{y}}

\newcommand{\bff}{\mathbf{f}}
\newcommand{\bfr}{\mathbf{r}}
\newcommand{\bfd}{\mathbf{d}}

\newcommand{\veps}{\vert \eps \vert}

\newcommand{\C}{\mathbb{C}}

\newcommand{\tOmega}{\widetilde{\Omega}}

\newcommand{\ri}{{\rm i}}

\newcommand{\beq}{\begin{equation}}
\newcommand{\eeq}{\end{equation}}
\newcommand{\beqs}{\begin{equation*}}
\newcommand{\eeqs}{\end{equation*}}
\newcommand{\bit}{\begin{itemize}}
\newcommand{\eit}{\end{itemize}}
\newcommand{\ben}{\begin{enumerate}}
\newcommand{\een}{\end{enumerate}}
\newcommand{\bal}{\begin{align}}
\newcommand{\eal}{\end{align}}
\newcommand{\bals}{\begin{align*}}
\newcommand{\eals}{\end{align*}}
\newcommand{\bse}{\begin{subequations}}
\newcommand{\ese}{\end{subequations}}
\newcommand{\bpr}{\begin{proposition}}
\newcommand{\epr}{\end{proposition}}
\newcommand{\bre}{\begin{remark}}
\newcommand{\ere}{\end{remark}}
\newcommand{\bpf}{\begin{proof}}
\newcommand{\epf}{\end{proof}}
\newcommand{\ble}{\begin{lemma}}
\newcommand{\ele}{\end{lemma}}
\newcommand{\bco}{\begin{corollary}}
\newcommand{\eco}{\end{corollary}}
\newcommand{\bex}{\begin{example}}
\newcommand{\eex}{\end{example}}
\newcommand{\bth}{\begin{theorem}}
\newcommand{\enth}{\end{theorem}}

\newcommand{\Rea}{\mathbb{R}}
\newcommand{\Com}{\mathbb{C}}

\newcommand{\supp}{\mathop{{\rm supp}}}

\newcommand{\eps}{\varepsilon}

\newcommand{\pdiff}[2]{\frac{\partial #1}{\partial #2}}
\newcommand{\dudn}{\pdiff{u}{n}}
\newcommand{\dudnw}{\partial u/\partial n}

\newcommand{\nrme}[1]{\Vert  #1 \Vert_{1,k}}

\newcommand{\Vell}{\cV_\ell}

\newcommand{\ksqeps}{\frac{k^2}{\vert\eps\vert}}

\newcommand{\gu}{\nabla u}

\newcommand{\gv}{\nabla v}

\newcommand{\LtG}{L^2(\GammaN)}

\newcommand{\LtO}{L^2(\Omega)}

\newcommand{\HoO}{H^1(\Omega)}

\newcommand{\tendo}{\rightarrow 0}

\newcommand{\matrixA}{A}
\newcommand{\matrixAeps}{{A}_\eps}
\newcommand{\matrixAepsinv}{\matrixAeps^{-1}}

\newcommand{\matrixAepsell}{A_{\eps, \ell}}

\newcommand{\matrixAzeroeps}{A_{\eps,0}}
\newcommand{\matrixAepszero}{A_{\eps,0}}

\newcommand{\bfzero}{\mathbf{0}}

\newcommand{\matrixR}{{ R}}

\newcommand{\matrixP}{{ P}}

\newcommand{\Qepsell}{Q_{\eps,\ell}}

\newcommand{\Qepsz}{Q_{\eps,0}}
\newcommand{\aeps}{a_\eps}

\def\XXint#1#2#3{{\setbox0=\hbox{$#1{#2#3}{\int}$}
     \vcenter{\hbox{$#2#3$}}\kern-.5\wd0}}

\newcommand{\bv}{\overline{v}}

\newcommand{\bgv}{\overline{\gv}}

\newcommand{\GammaN}{\Gamma}

\newcommand*{\N}[1]{\left\|#1\right\|}

\newcommand{\matrixS}{{ S}}
\newcommand{\matrixM}{{ M}}
\newcommand{\matrixB}{{ B}}
\newcommand{\matrixBepsinv}{{ B}_\eps^{-1}}
\newcommand{\matrixBepsAVE}{{ B}_{\eps,AVE}^{-1}}
\newcommand{\matrixBepsRAS}{{ B}_{\eps,RAS}^{-1}}

\newcommand{\matrixN}{{ N}}

\newcommand{\matrixC}{{C}}
\newcommand{\matrixI}{{ I}}

\definecolor{darkred}{RGB}{139,0,0}
\definecolor{darkgreen}{RGB}{0,100,0}
\definecolor{darkmagenta}{RGB}{139,0,139}
\definecolor{darkpurple}{RGB}{110,0,180}
\definecolor{darkblue}{RGB}{40,0,200}

\newcommand{\bfx}{\mathbf{x}}

\usepackage{hyperref}
\definecolor{myblue}{rgb}{0,0,0.6}
\hypersetup{colorlinks=true,
linkcolor=myblue,citecolor=myblue,filecolor=myblue,urlcolor=myblue}
\providecommand{\red}{}

\allowdisplaybreaks[4]

\newcommand{\ton}{\text{ on }}
\newcommand{\tin}{\text{ in }}
\newcommand{\tfa}{\text{ for all }}

\newcommand{\kseb}{\left(\frac{k^2}{\veps}\right)}
\newcommand{\eksb}{\left(\frac{\veps}{k^2}\right)}

\newcommand{\Hsub}{H_{\text{sub}}}

\begin{document}

\title[Domain Decomposition for high-frequency Helmholtz]{Domain Decomposition  preconditioning  for  \\  high-frequency 
Helmholtz problems \red{with} absorption }

\author{I.G. Graham}
\address{Dept of Mathematical Sciences, University of Bath,
Bath BA2 7AY, UK.}
\curraddr{}
\email{I.G.Graham@bath.ac.uk}
\thanks{}

\author{E.A. Spence}
\address{Dept of Mathematical Sciences, University of Bath,
Bath BA2 7AY, UK.}
\email{E.A.Spence@bath.ac.uk}
\thanks{} 

\author{E. Vainikko}
\address{Institute of Computer Science, University of Tartu, 50409, Estonia}
\email{eero.vainikko@ut.ee}
\thanks{}

\begin{abstract}
{In this paper we give new results on    
domain decomposition preconditioners  for GMRES when computing piecewise-linear  finite-element approximations} of the Helmholtz equation   
$-\Delta u  - (k^2+ \ri \eps)u = f$,  with absorption parameter $\eps \in \R$. Multigrid approximations of  this equation with $\eps \not= 0$ are  commonly used as preconditioners   for the 
 pure Helmholtz case ($\eps = 0$). 
However a  rigorous theory for such (so-called 
 ``shifted Laplace'') preconditioners, either for the pure Helmholtz equation,  or  
even the \red{absorptive} equation ($\eps \not=0$), is still missing. 
We present a new theory for the \red{absorptive} equation  
that provides 
rates of convergence for (left- or right-) preconditioned GMRES,   via estimates of  the norm and field of values of 
the preconditioned matrix. This theory  uses  a $k$- and $\eps$-explicit coercivity result for 
the underlying  sesquilinear form  and  shows, for example,  that if $|\eps|\sim k^2$,  then 
classical overlapping additive Schwarz  will perform  optimally 
for   the absorptive problem, provided  
the subdomain and coarse mesh diameters are   carefully chosen.   
\red{Extensive numerical experiments are given that  support the theoretical results.} 
While the theory applies to a  certain weighted variant of GMRES, the experiments for 
both weighted and classical GMRES give comparable  results.  
\red{The theory for the  absorptive case  gives insight into how  
its domain decomposition approximations 
perform as preconditioners for the pure Helmholtz case $\eps = 0$.}   
At the end of 
the paper we propose a (scalable)  multilevel preconditioner for the
pure Helmholtz  problem that has  
an empirical computation time complexity of about 
$\mathcal{O}(n^{4/3})$ for solving finite element systems of size  $n=\mathcal{O}(k^3)$, where we have chosen the mesh diameter $h \sim k^{-3/2}$ to avoid the pollution effect. Experiments on problems with $h\sim k^{-1}$, i.e. a fixed number of grid points per wavelength, are also given.
\end{abstract}


\maketitle

\section{Introduction}\label{sec:intro}
 
\noindent  
\red{This paper is concerned with domain-decomposition preconditioning for finite-element discretisations of the boundary value problem
\begin{equation}
\left\{  \begin{array}{rl}
- \Delta u  -  (k^2 + {\ri} \eps) u &=  f \quad \mbox{ in } \Omega,\\
\dudnw -  \ri \eta u &= g \quad \mbox{ on } \GammaN,   
\end{array}\quad \quad \right.\label{2}
\end{equation}
with $k>0$ and  $\eta = \eta(k,\eps)$, where either (i) $\Omega$ is a bounded domain in $\mathbb{R}^d$ 
with boundary $\Gamma$ or (ii) $\Omega$ is the exterior of a bounded 
scatterer,  $\GammaN$ denotes  an approximate   far field boundary, and the 
problem is appended with a homogeneous  Dirichlet condition on the boundary of the 
scatterer. }
\red{Although the PDE in \eqref{2} is relevant in  applications, our main motivation for studying this problem is its recent use in preconditioning the corresponding BVP for the Helmholtz equation:
\begin{equation}
\left\{  \begin{array}{rl}
- \Delta u -  k^2 u &=  f \quad \mbox{ in } \Omega,\\
\dudnw -  \ri k u &= g \quad \mbox{ on } \GammaN ,
\end{array}\quad \quad \right.\label{1}
\end{equation}
Linear systems arising from finite element
approximations of \eqref{2} 
with high wavenumber  
$k$ are notoriously hard to solve. Because the system matrices are 
non-Hermitian and generally non-normal, general iterative methods like 
preconditioned (F)GMRES have to be employed. Analysing the convergence of these 
methods is hard, since an analysis of the spectrum of the system matrix 
alone is not sufficient for any rigorous convergence estimates.}
 
\red{The idea of preconditioning discretisations of \eqref{1} with approximate  discretisations of \eqref{2} is often  called ``shifted  Laplacian" preconditioning. From its origins in  \cite{ErVuOo:04}, this  idea has had a large   impact on the field of practical  fast Helmholtz solvers. The main aim of the present  paper is to provide theoretical underpinning for this idea, and to use this theoretical understanding to   
develop new preconditioners for \eqref{1}. }
  
We denote the system matrix arising from 
continuous piecewise linear ($P1$)  Galerkin finite element approximations of \eqref{2} by
$A_\eps$ (or simply  $A$ when $\eps = 0$). 
For the solution of ``pure Helmholtz'' systems    
$A\mathbf{u} = \bff$, 
the ``shifted Laplacian'' preconditioning strategy (written in
left-preconditioning mode),  involves iteratively solving  the equivalent
problem 
 \beq\label{eq:discreteB}
\matrixBepsinv \matrixA \bu = \matrixBepsinv\bff, 
\eeq
where $B_\eps^{-1}$ is   some readily
computable  approximation   of  
$A_\eps^{-1}$ (for example a multigrid V-cycle). 
The rigorous analysis of the performance of this preconditioner is 
complicated, partly  because it is  based on a  double approximation:  
$A^{-1} \approx A_\eps^{-1} \approx B_\eps^{-1}$,    and partly because the 
\red{convergence theory of}   GMRES 
for non-self-adjoint systems requires one to estimate either the field of values of the system matrix or the spectrum and its conditioning.
  
One natural approach is to write 
\beq\label{eq:key_intro}
I -\matrixBepsinv\matrixA \ = \ \matrixI - \matrixBepsinv \matrixA_{\eps}
\ + \ \matrixBepsinv \matrixA_{\eps} (\matrixI - \matrixA_{\eps}^{-1} 
\matrixA)\ ,
\eeq
and to  recall that a sufficient (but by no means necessary)  condition
for GMRES to converge quickly is that the 
field of values of the system matrix should be bounded away from the origin and the norm of the system matrix should be bounded above. It is therefore clear from \eqref{eq:key_intro} that sufficient conditions  for  
$B_\eps^{-1}$ to be a good preconditioner for $A$ are: 
\begin{quotation} 
(i) $\matrixAepsinv$ is  a good
preconditioner for $\matrixA$   
\end{quotation} 
and 
 \begin{quotation} 
(ii) $\matrixBepsinv$ is  a good
preconditioner for $\matrixAeps$. 
 \end{quotation} 
\red{Achieving both (i) and (ii) simultaneously imposes contradictory requirements on $\eps$.}
Indeed, it is natural to expect that (i) holds if $|\eps|$ is sufficiently
small, but that for (ii) to hold we need $|\eps|$ sufficiently large.    
Most analyses of the performance of $B_\eps^{-1}$ as 
a preconditioner for $A$ have focused on obtaining conditions under which property (i) holds and \red{have}   
concentrated on analysing spectra. 
While a  detailed
literature survey is given in
\cite[\S1.1]{GaGrSp:13}, an up-to-date   summary of this  is given at the end of this section.

In \cite{GaGrSp:13} we   gave the first rigorous theory that 
identified  conditions that ensure 
(i) above holds. There, under general conditions on the domain and mesh sequence,
we showed that when $|\eps|/k$ was bounded above by a  sufficiently
small constant then (i) holds. 

The main theoretical purpose of the current paper is to obtain sufficient conditions for  (ii) to hold in the case when $B_\eps^{-1}$ is chosen as 
a   classical Additive Schwarz  preconditioner for 
$A_\eps$. 
We use the rigorous
convergence theory  of \cite{EiElSc:83} (see also \cite{CaWi:92}, \cite[\S1.3.2]{OlTy:14}),   in which criteria
for convergence of GMRES are given in terms of an  upper bound on
the norm of the system matrix and a lower bound on the distance of its
field of values from the origin. In the  Additive Schwarz construction,   the domain is covered with  
overlapping subdomains with   diameter denoted  $\Hsub$ and also  triangulated with 
a coarse mesh with diameter denoted  $H$. (It is not necessary for $\Hsub$ and $H$ to be related.)
The overlap parameter is denoted  
$\delta$, and  $\delta \sim H$ corresponds   to ``generous overlap''.   
Further  technical requirements are given in \S \ref{sec:DD}.

\red{We highlight at this stage that the conditions on $|\eps|$ that we find for (ii) above to hold do not overlap with those described above for (i) to hold, and thus the combination of this paper with \cite{GaGrSp:13} does not provide a complete theory for preconditioning the Helmholtz equation with absorption.
Nevertheless
\bit
\item[(a)] \red{we believe that} the present paper combined with \cite{GaGrSp:13} constitute the only rigorous results in the literature addressing when either of the properties (i) or (ii) above hold,
\item[(b)] the investigation into the property (ii) in the present paper, combined with the knowledge from \cite{GaGrSp:13} about the property (i), gives insight into how to design a good preconditioner for $A$ (albeit one currently without a rigorous convergence theory); this is especially true when considering multilevel methods -- see the discussion around Experiments 2 and 3 in \S\ref{sec:Numerical}.
\eit
}

\subsection{Summary of main theoretical results}
Throughout the paper we assume  that $\red{0<}\vert \eps\vert  \lesssim k^2$, so the ratio $\vert \eps\vert /k^2$ is always 
bounded above,   but may approach zero as $k \rightarrow \infty$. 
 Our main theoretical results  are Theorems   \ref{thm:final1}, and 
\ref{thm:final3} and their corollaries, which are proved in \S \ref{sec:Matrices}. 
Theorem \ref{thm:final1} examines the left-preconditioned matrix:  $B_\eps^{-1}A_\eps$, and obtains an upper bound on its norm and a lower bound on 
its field of values.
The upper bound on the norm is   
  $\cO(k^2/\vert\eps\vert)$, while the distance of the field of values from the origin (in the case 
of generous overlap)   has a lower bound of order   $\cO((\vert \eps \vert /k^2)^2)$.  
 These bounds are obtained subject to  the
subdomain and \red{coarse} mesh diameters satisfying bounds:   $k\Hsub \lesssim \vert \eps\vert /k^2$ and   
$kH \lesssim (\vert \eps\vert /k^2)^3$.

An important special case is that of maximum absorption  
$\vert \eps\vert  \sim k^2$. Then  the results imply  that the  number of GMRES iterates will   be 
bounded  independently  of $k$,  provided $\Hsub$ and $H$ both decrease with  order $k^{-1}$.
Thus,   provided     there is enough absorption in the system, the GMRES method will perform 
analogously to the Laplacian case, provided the   coarse mesh  decreases proportional  
to the wavelength (i.e.~no further refinement  for  pollution is needed). Thus if the 
mesh diameter $h$ of the fine grid decreases as $\cO(k^{-3/2})$ (needed to remove pollution in the underlying discretization), 
then considerable coarsening can be carried out. We actually 
see in the numerical experiments in \S \ref{sec:Numerical}  that further coarsening beyond the $k^{-1}$ theoretical limit \red{may be   possible, depending on the choice of $\eps$}.)          
Analogous results for right preconditioning (obtained by a duality argument) are 
given in Theorem \ref{thm:final3}. Then  Corollaries \ref{cor:final2} and \ref{cor:final4} give the  corresponding  estimates for GMRES convergence for each of these preconditioners. 
     
\subsection{How the theoretical results were obtained}        
 As in classical Schwarz theory, 
the proofs of Theorems \ref{thm:final1} and \ref{thm:final3} are obtained from  a 
projection operator analysis (given in \S \ref{sec:Convergence}).  However,   in order to get good results 
for large $k$  we do not use the classical approach of treating the Helmholtz operator as 
a perturbation of the Laplacian, \red{as was done in} \cite{CaWi:92} 
\red{(see also \cite{GoPa:03}, where this approach was used  for the time-harmonic Maxwell equations)}.
Rather,   we exploit  the coercivity of the 
problem with absorption  (Lemma \ref{lem:coer}),   leading  
to a  projection analysis in the 
wavenumber-dependent inner product $( \cdot , \cdot )_{1,k}$.  
The norm of the projection operator corresponding to the 
two-level algorithm  
is estimated above in Theorem \ref{thm:boundQ}, while the distance of its field
of values from the origin is estimated below in Theorem \ref{thm:final}. 
The analysis depends on a technical estimate on the approximation power of the coarse space (Assumption \ref{ass:approx}). We prove this estimate for convex polygons (Theorem \ref{thm:polygon}), and we also outline how to prove it for more general 2- and 3-d domains (Remark \ref{rem:approx}).

The estimates for the projection operators in \S \ref{sec:Convergence} are converted to estimates for the norm and field of values of preconditioned Helmholtz matrices in \S \ref{sec:Matrices}. 
\red{Because} the analysis is performed in the ``energy'' inner product $\Vert \cdot \Vert_{1,k}$, the corresponding matrix estimates are obtained in the induced  weighted Euclidean inner product. (A similar situation arises in the classical analysis \cite{CaZo:02}.)  We performed  numerical experiments both for standard GMRES (with residual mininmization in the Euclidean norm) and for weighted GMRES (minimizing in the weighted norm), but in practice there was little difference in the results. 
    
\subsection{Overview of numerical results}    
A sequence of  numerical experiments is given in  \S\ref{sec:Numerical} for solving systems \red{with matrix} $A_\eps$ 
with $h\sim k^{-3/2}$ ($n \sim k^3$, where $n$ is the system dimension), yielding  (empirically) pollution-free finite element  solutions.   
\red{In these experiments $H\sim \Hsub$.}
First, we  consider the performance of the
preconditioner  $B_\eps^{-1}$ (defined by the  classical
Additive Schwarz method), when   applied to problems with 
coefficient  matrix $A_\eps$. 
As predicted by  the theory,   we see that 
$B_\eps^{-1} $ is an optimal preconditioner when  $|\eps| \sim
k^2$ (i.e.~the number of GMRES iterates is parameter independent), provided  the coarse grid
diameter $H$ and subdomain diameter $\Hsub$  are  sufficiently
small. Experimentally, good results are also obtained even with larger $H,
\Hsub$ when $\eps$ is large
enough, and even with smaller  $\eps$  when $H, \Hsub$ are small enough. 
We also test variants of the classical method, 
including Restricted Additive Schwarz (RAS) and the Hybrid
variant of this (HRAS) (where coarse and local parts of the preconditioner
are combined multiplicatively). Out of all the methods tested, HRAS
performs the best. 

\red{Based on this empirical insight gained about preconditioning $A_\eps$,}
we then investigate the performance of HRAS
(with absorption $\eps$) as a preconditioner for the pure Helmholtz
problem  with coefficient matrix  $A$. We
find that HRAS still works well,  provided $H$ and $\Hsub$ are small
enough. There is  surprisingly little variation in the performance with
respect to the choice of $\eps$. (In fact with $|\eps| = k^\beta$, the
performance is almost uniform in the range $\beta \in [0,1.2]$ but
there is some degradation as $\beta$ approaches $2$.   This is
surprising as the choice $\beta \sim  2$ is normally used in the multigrid context.
We also test a variant of
HRAS that uses impedance conditions on subdomain solves and this
works well,  especially for larger $H, \Hsub$. 

Finally,  to solve problems with
matrix $A$ in the case of large $k$,  we  recommend an inner-outer preconditioner for use
within FGMRES, where the outer solver is HRAS with $|\eps| = k$ and
$H\sim \Hsub \sim k^{-1}$.  The cost of the preconditioner is then dominated
by the coarse grid problem, and for this we apply an inner iteration
with preconditioner chosen as one-level HRAS with impedance boundary
condition on local problems. With the best choice of $\eps$ appearing
to be $|\eps| \sim k$, we find this solver has a compute time of about
$\mathcal{O}(k^4) \sim \mathcal{O}(n^{4/3})$ for the 2D problems
tested, up to $k = 100$.
This is a highly scalable preconditioner, whose action consists of inverting $\cO(k^2)$ (parallel) finite-element systems of size $\cO(k)$ and an additional $\cO(k)$ finite-element systems of size $\cO(k)$.
\red{Additional experiments, together with multilevel  variations suitable for the case $h \sim k^{-1}$ are given in \cite{GrSpVa:16}.}

\subsection{Literature review} 
We finish this section with a short literature survey on this topic, beginning with the literature on preconditioning with absorption, and then briefly discussing domain decomposition methods for wave problems.

The survey in \cite[\S1.1]{GaGrSp:13} focused on the spectral analyses in \cite{ErVuOo:04}, \cite{ErOoVu:06}, \cite{VaErVu:07}, \cite[\S5.1.2]{Er:08}, \cite{ErGa:12}, all of which concern the optimal choice of $\eps$ for $A_\eps$ to be a good preconditioner for $A$, \red{i.e. for property (i) above to hold}.
Several authors have considered the \red{question of when} multigrid methods \red{converge when applied} to the problem with absorption (i.e. $A_\eps$); \red{this is related to (but not the same as) the question of when property (ii) above holds.}
Cools and Vanroose \cite{CoVa:13} computed the ``minimal shift" (defined as the smallest value of $\eps$ for which every single eigenmode of the error is reduced through consecutive multigrid iterations) based on numerical evaluation of quantities arising from Fourier analysis, and found that (as a function of $k$) it is proportional to $k^2$.  Cocquet and Gander \cite{CoGa:14} (following on from \cite{ErGa:12}) 
showed that, for a particular standard variant of multigrid applied to the 1-d Helmholtz equation with Dirichlet boundary conditions, one
 needs $|\eps|\sim k^2$ to obtain convergence independent of $k$. They also analysed a less-standard variant of multigrid applied to general multi-dimensional Helmholtz problems with either Dirichlet or impedance boundary conditions, and showed that again \red{one} needs $\red{|\eps|}\sim k^2$ for the method to be practical. 
Note that these analyses are concerned with the {\em convergence} of 
multigrid as a solver for $A_\eps$, rather than using an approximation such as the V-cycle applied to the problem with absorption  as 
a preconditioner for either the absorbing or the non-absorbing problem ($A$ or $A_\eps$ respectively).

The study of non-overlapping domain-decomposition methods for wave problems has a long history, starting with the seminal
paper of Benamou and Despr\'es \cite{BeDe:97}. Following that, 
 optimized interface conditions were introduced \cite{GaMaNa:02},  the 
 success of which  sparked substantial interest, for
 example \cite{GaHaMa:07}, \cite{DoGaGe:09}, and more recently the ``source
 transfer'' and related methods \cite{ChXi:13}, \cite{ChXi:13a}, and \cite{St:13}; these latter methods can be viewed as putting the ``sweeping" method of \cite{EnYi:11b} in a continuous (as opposed to discrete) setting.
 All these non-overlapping domain decomposition methods focus on the choice
 of good interface conditions but so far do not provide a systematic method of combining these with coarse grid operators or a convergence analysis explicit in subdomain or coarse
 grid size.    There are  also a   few results on overlapping domain
 decomposition methods e.g.     \cite{To:98}, \cite{KiSa:07}, \cite{KiSa:13}, with the latter
 explicitly using absorption; these demonstrated the potential of the
 methods analysed in this paper.
 Finally, we note that \cite{ZeDe:14} introduces a new sweeping-style method for the Helmholtz equation, and also contains a good literature review of both domain-decomposition and sweeping-style methods.

\section{Variational formulation}

\label{sec:Var} 

For ease of exposition,
we restrict attention to the interior impedance problem (i.e.~\eqref{1} is posed for $\Omega$ a bounded domain in $\mathbb{R}^d$ 
with boundary $\Gamma$). The results of the paper also hold for the truncated sound-soft scattering problem, and we outline in Remark \ref{rem:trunc} how to adapt them to this case. 

Let $\Omega$ be a bounded, open, 
polygonal (Lipschitz polyhedral) 
domain in  
$\mathbb{R}^d$, $d = 2$ (or $3$),  with boundary $\Gamma$. 
We introduce the standard   
$k$-weighted inner product and norm on $H^1(\Omega)$:  
$$
(v,w)_{1,k} \ = \ (\nabla v, \nabla w)_{L^2(\Omega)} + k^2 (v,w)_{L^2(\Omega)} \quad  \text{and}  \quad 
\Vert v \Vert_{1,k} = (v,v)_{1,k}^{1/2} \ .  
$$
The standard variational formulation of \eqref{2} is:
Given $f \in L^2(\Omega)$, $g\in \LtG$, $\eps \in \mathbb{R}$ and $k>0$ find $u \in \HoO$ such that 
\beq\label{eq:vf_intro}
a_\eps(u,v) = F(v) \quad \text{ for all }\,  v \in \HoO,
\eeq
where 
\beq\label{eq:Helmholtzvf_intro}
a_\eps(u,v) := \int_\Omega \gu \cdot \bgv - (k^2 + \ri \eps) \int_\Omega u \bv - \ri \eta  \int_{\GammaN} u\bv, 
\eeq
and 
\beq\label{eq:F_intro}
F(v) := \int_\Omega f\bv + \int_{\GammaN} g \bv.
\eeq
In general $\eta$  can be complex, with a natural choice being a square root of $k^2 + \ri \eps$. (more details are in Lemma \ref{lem:coer}). 
When $\eps = 0$ and $\eta = k$ we are solving \eqref{1} and we simply write 
$a$ instead of $a_\eps$. 

We consider the discretisation of  problem \eqref{eq:vf_intro} with $P1$ 
 finite elements.   
Let 
 $\cT^h$ be  a family of
conforming  meshes  (triangles in 2D, tetrahedra in 3D), 
that are shape-regular as the  
mesh diameter  $h \rightarrow 0$.  
A typical element of $\cT^h$ is  $\tau \in \cT^h$ (a closed subset of
$\overline{\Omega}$). 
Then our approximation space $\cV^h$ is the space of all continuous 
functions on $\Omega$ that are piecewise affine with respect to $\cT^h$. (The impedance boundary condition in \eqref{1} is 
implemented as a natural boundary condition.) 
The freedoms for  $\cT^h$ are the nodes,   denoted 
$\cN^h = \{x_j:j \in \cI^h\}$, where $\cI^h$ is  a suitable index
set. 
The standard basis for $\cV^h$ is 
$\{ \phi_j : j \in \cI^h \}$ 
consisting of hat functions corresponding to the each of the nodes in $\cN^h$. 

The Galerkin approximation of \eqref{eq:vf_intro} in the space 
$\cV^h$ is equivalent to the  system 
 \beq\label{eq:discrete}
 \matrixA_\eps \bu  \ := \ (\matrixS - (k^2+ \ri \eps)  \matrixM - \ri \eta \matrixN) \bu \ =\  \bff , 
 \eeq
where   
\begin{equation}\label{eq:matrices}
\matrixS_{\ell,m} = \int_{\Omega} \nabla \phi_\ell \cdot \nabla
\phi_m , \quad   \matrixM_{\ell,m} = \int_{\Omega}
\phi_\ell  \phi_m , \quad     
\matrixN_{\ell,m} = \int_{\Gamma} \phi_\ell  \phi_m, \quad  
\ell, m \in \cI^h\ \end{equation}
are,  respectively,  the stiffness matrix, the domain mass matrix,
and the boundary
mass matrix.  Again we write the corresponding system matrix for \eqref{1} simply as $A$. Note that $\matrixA$ and $\matrixA_\eps$ are symmetric
but not Hermitian. 

In this section we briefly provide the key properties of the
sesquilinear form $a_\eps$ given in \eqref{eq:Helmholtzvf_intro}. This
form depends on all three  parameters $\eps, k$ and $\eta$, but only the first of these is reflected in the notation.  Normally $\eta$ will be chosen as a function of $\eps$ and $k$. We will assume throughout that
\begin{equation} \label{eq:est_epseta} 
\vert \eps\vert  \lesssim  k^2  \quad \text{and} \quad \vert \eta \vert \lesssim k . 
\end{equation} 
(Here the notation $A\lesssim B$ (equivalently $B \gtrsim A$) means that $A/B$ is bounded above by a constant independent of $k$, $\eps$, and mesh diameters $h, \Hsub, H$ (the latter two introduced below).  We write $A \sim B$ when  $A\lesssim B$ an  $B \lesssim  A$. 

The proof of the first result is  a simple application of the  
Cauchy-Schwarz  and  multiplicative trace 
inequalities - see,  e.g., \cite[Lemma 3.1(i)]{GaGrSp:13}. 
 \ble[Continuity]\label{lem:cont}
If $\vert \eta \vert \lesssim k$ then,  
given $k_0>0$, there exists a $C_c$ independent of $k$ and $\eps$ such that 
\beq
|a_{\eps}(v,w)| \leq C_c \N{v}_{1,k} \N{w}_{1,k}
\label{eq:conty} 
\eeq
for all $k\geq k_0$ and $v, w \in H^1(\Omega)$. 
\end{lemma}

We now give a result about the coercivity of $a_\eps$, which is a generalisation of \cite[Lemma 3.1(ii)]{GaGrSp:13}. To state this we need to define $\sqrt{k^2 + \ri \eps}$, taking care to cater for both positive and negative $\eps$. We need to consider both positive and negative $\eps$ since, whichever choice we make for the problem \eqref{2}, the other forms the adjoint problem, and we need estimates on the solutions and sesquilinear forms for both problems (in particular, this is essential for analysing both left- and right-preconditioning).

\begin{definition}\label{def:z}
$z(k,\eps):= \sqrt{k^2 + \ri \eps}$ is defined via the square root with the branch cut on the positive real axis. Note that this definition implies that, when $\eps\neq 0$,
\beq
\label{eq:propsqrt}
\Im(z) > 0, \quad \mathrm{sign}(\eps) \Re(z)>  0, \quad \text{ and }\quad z(k,-\eps) = - \overline{z(k,\eps)}.
\end{equation}
\end{definition}

\begin{proposition}\label{prop:elem}
With $z(k,\eps)$ defined above, for all $k>0$,
\begin{equation}\label{eq:ests4lemma}
\vert z \vert \sim k \quad \text{and} \quad  \frac{\Im(z)}{\vert z\vert} \sim \frac{\vert \eps\vert }{k^2}.
\end{equation}
\end{proposition}

\begin{proof} 
Writing  $z = p+ \ri q$, we see that the definition of $z$ implies that 
\beqs
p= \sqrt{p^2} \,\, \text{if }\,\,\eps>0, \quad p = - \sqrt{p^2} \,\,\text{if}\,\, \eps<0, \quad \text{and} \quad q = \sqrt{q^2} \, \tfa \eps\neq 0, 
\eeqs
where
\beq
p^2= \frac{\sqrt{k^4+\eps^2} + k^2}{2}, \   \quad\mbox{ and } \quad
q^2= \frac{\sqrt{k^4+\eps^2}- k^2 }{2}
\label{eq:formulae} 
\eeq
(and $\sqrt{\cdot}$ denotes the positive real square root).
Using \eqref{eq:est_epseta} we therefore see that $|p|\sim k$. Furthermore, the definition of $z$ implies that $2pq = \eps$, and thus $q =|q|\sim |\eps|/|p| \sim |\eps|/k$. Using  \eqref{eq:est_epseta} again,  
the estimates \eqref{eq:ests4lemma} follow.
\end{proof}

\begin{lemma}[Coercivity]\label{lem:coer}
Let $z= z(k,\eps)$ be as defined in Definition \ref{def:z}, and
choose     $\eta$ in \eqref{eq:Helmholtzvf_intro} to satisfy the inequality
\begin{equation} \Re (\overline{z} \eta) \geq 0 .  \label{eq:condition} \end{equation}
Then there is  
a constant $\rho>0$ independent of $k$ and $\eps$ such that 
\beq\label{eq:coercivity}
\vert a_{\eps}(v,v) \vert \ \geq \ \Im\left( \Theta a_\eps(v,v) \right)\ \geq \ \rho \, \frac{\vert \eps\vert }{k^2}\,  \nrme{v}^2
\eeq
for all $k>0$ and $v\in H^1(\Omega)$,  where $\Theta = -\overline{z}/\vert z \vert$.
\end{lemma}

\begin{proof}
Writing $z = p+iq$ and using 
the definition of $a_\eps$,  we have 
$$ a_{\eps}(v,v) \ = \ \Vert \gv \Vert_{L^2(\Omega)}^2 - (p+\ri q )^2 \Vert v \Vert_
{L^2(\Omega)}^2  - \ri \eta \Vert v \Vert_{L^2(\Gamma)}^2  \ .
$$
Therefore
\begin{equation*}
\Im \left[ -(p-\ri q)  a_{\eps}(v,v)\right] \  =\  q \Vert \gv \Vert_{L^2(\Omega)}^2 +  
q (p^2+ q^2) \Vert v \Vert_
{L^2(\Omega)}^2  + \Re\left[  (p - \ri q)\eta \right] \Vert v \Vert_{L^2(\Gamma)}^2  \ . 
\end{equation*}
Hence, dividing through by $\vert z \vert = \sqrt{p^2 + q^2} $, and setting 
$\Theta = - \overline{z} /\vert z \vert$,  we have 
\begin{equation*}
\Im \left[ \Theta   a_{\eps}(v,v)\right] \  =\  \frac{\Im(z)}{\vert z\vert} \left[ \Vert \gv \Vert_{L^2(\Omega)}^2 +  
 \vert z \vert^2  \Vert v \Vert_
{L^2(\Omega)}^2\right]  + \frac{\Re\left(\overline{z} \eta \right)}{|z|} \Vert v \Vert_{L^2(\Gamma)}^2  \ .  
\end{equation*}
The result then follows from condition \eqref{eq:condition} and
the second estimate in  \eqref{eq:ests4lemma}. 
\end{proof} 
\begin{remark}[Choices of $\eta$ satisfying \eqref{eq:condition}] \label{rem:obvious} 
An obvious choice of $\eta$ that satisfies the coercivity
  condition \eqref{eq:condition} is $\eta = z$, for then 
$ \Re(\overline{z}\eta) = \Re(\overline{z} z) = \vert z \vert ^2  > 0
$, Another possible choice is $\eta = \mathrm{sign}(\eps) k$, for
then, by \eqref{eq:propsqrt}, we have
$ \Re(\overline{z}\eta) = \mathrm{sign}(\eps) \Re(z) k  >0$.   
Note that both these choices satisfy the condition on $|\eta|$ in \eqref{eq:est_epseta}.

The fact that the choice of $\eta$ for coercivity to hold depends on the sign of $\eps$ is expected, since the sign of $\eps$ also dictates the properties of $\eta$ required for the problem \eqref{2} to be well posed. Indeed, repeating the usual argument involving Green's identity (given for $\eps=0$ in, e.g., \cite[Theorem 6.5]{Sp:15}) we see that if $\eps>0$ we need $\Re(\eta)\geq 0$ for uniqueness and if $\eps<0$ we need $\Re(\eta)\leq 0$.

The condition for coercivity \eqref{eq:condition} is more restrictive that the conditions for uniqueness. Indeed, since $\Re(\overline{z} \eta) = \Re(z) \Re(\eta) + \Im(z) \Im(\eta)$,  
when $\eps> 0$, a sufficient condition to ensure \eqref{eq:condition} is $\Re(\eta) > 0, \ \Im(\eta) \geq 0$. Similarly, when   $\eps<  0$  a 
sufficient condition for \eqref{eq:condition} is   $\Re(\eta) < 0, \ \Im(\eta) \geq 0$.      
\end{remark} 

This lemma immediately also gives us a result about the coercivity of
the  adjoint of $a_\eps$,  given by 
$$ a^*_\eps(u,v) \ = \int_\Omega \gu \cdot \bgv - (k^2 - \ri \eps) \int_\Omega u \bv +  \ri \overline{\eta}  \int_{\GammaN} u\bv. 
$$
\begin{corollary} \label{cor:adjoint}
Under assumption \eqref{eq:condition} we also have coercivity of the
adjoint form:  
\beq\label{eq:adjcoercivity}
\vert a_{\eps}^*(v,v) \vert \ \geq \ \Im\left( \Theta a_\eps(v,v) \right)\ \geq \ \rho \, \frac{\vert \eps\vert }{k^2}\,  \nrme{v}^2
\eeq
for all $k>0$ and $v\in H^1(\Omega)$,  where $\Theta = -\overline{z}/\vert z \vert$.
\end{corollary} 
\begin{proof}
Note that the adjoint form   is
simply a copy of the original form $a_\eps$, but  with parameters $\eps$ and 
$\eta$ replaced by   $\widetilde{\eps} = -\eps$ and  $\widetilde{\eta} = -
\overline{\eta}$, and thus (by \eqref{eq:propsqrt}) $\widetilde{z} = -\overline{z}$.
The condition
for coercivity of the adjoint form is then
$\Re (\overline{\widetilde{z}}\widetilde{\eta} ) \geq0$, which is equivalent to
condition \eqref{eq:condition}.  
\end{proof} 

\begin{remark}\label{rem:Throughout}
Throughout the paper we will always assume that $\eps$ and $\eta$ are  chosen so
that conditions \eqref{eq:est_epseta} and  \eqref{eq:condition} hold, 
and so the forms $a_\eps$ and
$a_\eps^*$ always will satisfy the continuity and coercivity estimates 
\eqref{eq:conty}, \eqref{eq:coercivity} and \eqref{eq:adjcoercivity}.

\end{remark}

\section{Domain Decomposition}

\label{sec:DD}

To define appropriate subspaces of  $\cV^h$,  
we start with a collection of open subsets    $\{ \tOmega_\ell: \ell =
1, \ldots, N\}$ of $\R^d$  that form an overlapping cover of $\overline{\Omega}$, 
and we set $\Omega_\ell =
\tOmega_\ell \cap \overline{\Omega}$. Each $\overline{\Omega}_\ell$  is assumed to be non-empty 
and $\overline{\Omega}_\ell$ is assumed to  consist of a union of
elements of the mesh $\cT_h$. Then, for each $\ell = 1, \ldots, N$,  we set 
$$\cV_\ell = \{ v_h \in \cV^h: \supp(v_h) \subset
\overline{\Omega}_\ell\}. $$  Note that, since functions in $\cV^h$
are continuous, functions in  $\cV_\ell$ must vanish  on the
internal boundary $\partial \Omega_\ell \backslash \Gamma$,
but are unconstrained on the
external boundary $ \partial \Omega_\ell\cap \Gamma$.
The freedoms for  $\cV_\ell$ are denoted   
$\cN^h(\Omega_\ell) = \{x_j:j \in \cI^h(\Omega_\ell)\}$, where $\cI^h(\Omega_\ell)$ is a suitable
index set. The  
basis for
$\cV^h(\Omega_\ell)$ can then be  written   $\{ \phi_j : j \in \cI^h(\Omega_\ell)
\}$.

Thus a solve of the Helmholtz problem \eqref{eq:vf_intro}  in the space $\cV_\ell$ involves a
Dirichlet 
boundary condition at internal boundaries and natural boundary
condition at external boundaries (if any). The introduction of the absorption
$\eps \not = 0$ ensures such solves are always  well-defined.  Future work
will consider  the analysis of methods with other local boundary
conditions (such as impedance or PML).  Internal impedance
conditions are considered in the experiments in \S \ref{sec:Numerical}.  

For      
$j \in \cI^h({\Omega_\ell})$ and $j' \in \cI^h$, we define the
restriction matrix  
$(R_\ell)_{j,j'} := \delta_{j,j'}$. The matrix 
$A_{\eps,\ell} \ := \ R_\ell  A_\eps R_\ell ^T$
is then  just the minor  of 
$A_\eps$ corresponding to rows
and columns taken from $\cI^h({\Omega_\ell})$.  One-level domain
decomposition methods are constructed from the inverses $\matrixAepsell^{-1}$. 
More precisely, 
\begin{equation}\label{eq:minor} 
 B_{\eps, {AS}, local}^{-1} \ : = \  \sum_{\ell}R_\ell^T A_{\eps,\ell}^{-1} R_\ell
\end{equation}
 is the classical
one-level preconditioner for $A_\eps$ with the subscript  ``$local$''
indicating  that the solves are on local subdomains $\Omega_\ell$. 

For the theory, we need  assumptions on the shape of the subdomains and
the size of the overlap, and we require any point in $\overline{\Omega}$ to  belong to a
bounded number of overlapping subdomains.   
First, for simplicity we assume the subdomains are shape-regular
Lipschitz polyhedra (polygons in 2D) of diameter $H_\ell = \mathrm{diam}
  (\Omega_\ell)$, with the volume of order    $\sim H_\ell^d$ and surface
  area 
  $\sim H_\ell^{d-1}$ respectively.  
The coarse mesh diameter $\Hsub : = \max\{  H_\ell : \ell = 1, \ldots,
N\}$ is then a parameter in our estimates.
Each $\Omega_l$ is required to have a large enough interior boundary, i.e.~we require that
\beq\label{eq:int_bound}
|\partial \Omega_l \setminus \Gamma| \sim \Hsub^{d-1} \quad\text{ for each } \, l.
\eeq

Concerning the overlap, for each $\ell = 1, \ldots , N$, let $\mathring{\Omega}_\ell$ denote the
part of $\Omega_\ell$ that is not overlapped by any other subdomains,
and for $\mu>0$ let $\Omega_{\ell, \mu}$ denote the set of points in
$\Omega_\ell$ that are a distance no more than $\mu$ from the
boundary $\partial \Omega_\ell$. Then we assume that for some
$\delta>0$ and some $0<c<1$ fixed, 
  \begin{equation}\label{eq:unifoverlap}\Omega_{\ell, c \delta}\subset \Omega_\ell\backslash
  \mathring{\Omega}_\ell \subset \Omega_{\ell, 
\delta}.\end{equation}
Put more simply, the overlap is assumed to be uniformly of order
$\delta$; 
the case $\delta\sim H$ is called ``generous overlap".
Finally,  we  make the {\em finite
  overlap assumption} 
\begin{equation} 
\#
  \Lambda(\ell) \lesssim 1, \quad \text{where} \quad   
\Lambda(\ell) = \{ \ell' :  \Omega_\ell \cap \Omega_{\ell'} \not =
\emptyset \} \ .   \label{eq:finoverlap}\end{equation}

Two-level methods are obtained by adding a  global coarse solve.  
Let  $\{\cT^{H}\}$ be  a sequence of shape-regular, simplicial meshes on 
$\overline{\Omega}$, \red{with mesh diameter $H$}. We assume that each element of $\cT^H$ consists of the union of a set of 
fine grid elements. The set of 
 coarse mesh nodes is
denoted by $\cI^H$.
The coarse space basis
functions  $\Phi_p$ are taken to be the continuous $P1$ hat functions
 on   $\cT^H$. 

From these functions we define 
the coarse space
$
\cV_0 \ := \ \mathrm{span} \{ \Phi_p : p \in \cI^{H}
\}\ , 
$
which is  a subspace
of $\cV^h$. 
Now, if we introduce the restriction matrix 
\begin{equation}\label{eq:restriction}(R_0)_{pj} \ := \ \Phi_p(x_j^h)\ ,  \quad j \in \cI^h , \quad
p \in \cI^{H},\end{equation}
then the matrix 
\begin{equation}\label{eq:coarsegrid}
\matrixAzeroeps := R_0 \matrixAeps R_0^T\ \end{equation}
is the stiffness matrix for problem (\ref{1})
discretised in $\cV_0$ using the basis $\{\Phi_p: p \in
\cI^{H}\}$.  Note that, due to the coercivity 
result Lemma \ref{lem:coer}, both $\matrixAzeroeps$ and $\matrixA_{\eps,\ell}$ are invertible for all mesh 
sizes $h$ and all choices of  $\epsilon \not = 0$. 
This is easily seen, since,  for example,  if  
$\matrixAepszero \bf v
= \mathbf{0}$, where $\mathbf{v}$ is a vector defined on the freedoms
$\cI^H$,   then
$0 \ = \ \bfv^*\matrixAepszero \bfv \ = \ a_\eps(v_H, v_H) $, 
where $v_H = \sum_{p \in \cI^H} v_p \Phi_p$ and so
$$0 \ = \ \vert a_\eps(v_H, v_H) \vert  \ \geq \ \rho \frac{\vert \eps\vert }{k^2} \Vert v_H\Vert_{1,k}^2, $$
which immediately implies $v_H = 0$, and thus $\bfv =
\bfzero$. Similar arguments apply to $\matrixAepsell$ and to the
adjoints $A_{\eps,\ell}^*$, $\ell = 0, \ldots, N$.  

The classical Additive
Schwarz method is
\begin{equation}\label{eq:defAS}
\matrixB_{\eps,AS}^{-1} \ : = \  \matrixR_0^T \matrixA_{\eps,0}^{-1} \matrixR_0 \ + \  \matrixB_{\eps,AS,local}^{-1} , 
\end{equation}
(i.e.~the sum of  coarse solve and local solves) with
$B_{\eps,AS,local}^{-1}$ defined in \eqref{eq:minor}.  
 
\section{Theory  of Additive Schwarz Methods}
\label{sec:Convergence}
The following theory establishes rigorously the powerful properties of the
preconditioner \eqref{eq:defAS} applied to $A_\eps$ if $\vert \eps\vert $ is   sufficiently  large and $\Hsub$,
$H$ are sufficiently small. 

This theory was inspired by reading again the results in \cite{CaWi:92} where non-self-adjoint problems that were ``close to" self-adjoint coercive problems were considered. Although our problem here is not close to a self-adjoint coercive one, and our technical tools are very different, \cite{CaWi:92} provided a framework that we were able to adapt into the following results.

The first lemma is an extension of the familiar  
``stable splitting'' property of domain decomposition spaces. This is well-known for the $H^1$ norm (see, e.g., \cite{ToWi:05})  but here we extend it 
to  the case of the $k$-weighted energy norm.

\begin{lemma}\label{lem:stable_splitting}
For all $v_h \in \cV^h$ , there exist $v_{\ell} \in \cV_\ell$ 
for each $\ell = 0, \cdots, N$ such that 
\begin{equation}
\label{eq:stable_splitting}
v_h = \sum_{\ell=0}^N v_{\ell} \quad \text{and } \quad 
\sum_{\ell=0}^N \Vert v_{\ell}\Vert_{1,k}^2 \ \lesssim\ \left( 1 + \frac{H}{\delta} \right) \Vert v_h \Vert_{1,k}^2\ .
\end{equation}  
\end{lemma}
\begin{proof}
This is adapted from  the proof of analogous results for 
Laplace problems; see, e.g., \cite{ToWi:05}.  The proof starts by approximating $v_h$ by the quasiinterpolant from the coarse space:
$$v_0 : = \sum_{p\in \cI^H} \widehat{v}_p \Phi_p^H $$
where $$\widehat{v}_p = \vert \omega_p\vert^{-1} \int_{\omega_p} v_h \quad \text{and} \quad \omega_p = \mathrm{supp} (\Phi_p^H)\ . $$ 
Then, using the shape regularity of $\cT^H$ it is straightforward to show that 
\begin{equation}
\label{eq:stable5} 
\Vert v_0 \Vert_{L^2(\Omega)} \lesssim \Vert v_h \Vert_{L^2(\Omega)}. 
\end{equation}
Next, we
take  a partition of unity $\{\chi_\ell: \ell = 1, \ldots , N\}$ 
subordinate to the covering 
$\Omega_\ell$ and  set 
\begin{equation}\label{eq:sp2level}
v_\ell = I^h(\chi_\ell(v_h - v_0))\ ,\end{equation} 
where $I^h$ denotes nodal interpolation onto $\cV^h$. 
The first equality in   
\eqref{eq:stable_splitting} follows easily after summation. 
Moreover the estimate
\begin{equation}\sum_{\ell=0}^N \vert v_{\ell}\vert_{H^1(\Omega)}^2 \ \lesssim\ \left( 1 + \frac{H}{\delta} \right) \vert v_h \vert_{H^1(\Omega)}^2\  \label{eq:stable3} 
\end{equation}
is familiar from results on self-adjoint coercive problems; see, e.g.,
 \cite[Theorem 3.8]{GrLeSc:07}.

 To obtain the second inequality in \eqref{eq:stable_splitting},  
 we note that by definition of $I^h$, we have, for any
$\tau \in \cT^h$ with $\tau \subset 
\overline{\Omega_\ell}$, and any  
$x \in \tau$, we have 
\begin{align*}
\vert v_\ell(x) \vert \ = \ & \left\vert \sum_{j \in \cI^h(\tau)}
(\chi_\ell (v_h- v_0))(x_j^h) \phi_j^h(x)\right\vert \ \leq \  \sum_{j \in \cI^h(\tau)} \vert (v_h- v_0)(x_j^h)\vert \\
 \ \lesssim\  & 
\left\{\sum_{j \in \cI^h(\tau)} \vert (v_h- v_0)(x_j^h)\vert^2\right\}^{1/2}
\ \sim \ \vert\tau\vert^{-1/2} \Vert v_h - v_0 \Vert_{L^2(\tau)}, 
\end{align*}
where $\{x_j : j \in \cI^h(\tau)\}$ denotes the nodes on $\tau$.
Hence 
\begin{align}
\Vert v_\ell \Vert_{L^2(\Omega)}^2  \  = \ & \ \sum_{\tau \subset
  \overline{\Omega_\ell}} \int_{\tau} \vert v_\ell\vert^2 \ \leq
\ \sum_{\tau \subset{\overline{\Omega_\ell}}}\vert \tau \vert \Vert v_\ell
\Vert_{L^\infty(\tau)}^2 \nonumber \\ \   \lesssim\  & \sum_{\tau \subset \overline{\Omega_\ell}} \vert \tau\vert \vert \tau\vert^{-1}  \Vert v_h- v_0 \Vert_{L^2(\tau)}^2      \  = \  \Vert v_h- v_0 \Vert_{L^2(\Omega_\ell)}^2 . \label{eq:compare}
\end{align}
Thus,  because of the finite overlap property \eqref{eq:finoverlap}, 
we have 
\begin{equation}
\sum_{\ell = 1}^N \Vert v_\ell \Vert_{L^2(\Omega)}^2  \ 
     \  \lesssim  \  \Vert v_h - v_0 \Vert_{L^2(\Omega)}^2 \ \lesssim \  
\Vert v_h \Vert_{L^2(\Omega)}^2 + \Vert v_0 \Vert_{L^2(\Omega)}^2 \ .  
\label{eq:stable2} \end{equation}
Combination of this with \eqref{eq:stable5} yields 
$\sum_{\ell=0}^N \Vert v_\ell \Vert_{L^2(\Omega)}^2 \lesssim \Vert
v_h \Vert_{L^2(\Omega)}^2 \ .  $  Then multiplication by $k^2$ and
combination with  \eqref{eq:stable3} gives  the required result.  
\end{proof}

The next  lemma is a kind of converse to Lemma \ref{lem:stable_splitting}. 
Here the energy of a sum of components is estimated above by the sum of the energies. 

\begin{lemma}\label{lem:norm_of_sum}
For all choices of $v_{\ell} \in \cV_{\ell}$ ,  
 $\ell = 0, \cdots, N$, we have  
\begin{equation}
\label{eq:norm_of_sum}
\quad \left\Vert \, \sum_{\ell=0}^N  v_{\ell} \, \right\Vert_{1,k}^2 \ \lesssim\ 
\sum_{\ell=0}^N \Vert v_{\ell} \Vert_{1,k}^2\ .
\end{equation}  
\end{lemma}
\begin{proof}
Let  $\displaystyle{\sum_\ell}$ denote  the sum from $\ell = 1$ to $N$ and  
and recall the notation   $\Lambda(\ell)$ introduced in
\eqref{eq:finoverlap}.  
 Then, using  several applications of the Cauchy-Schwarz inequality,
\begin{align}
\left\Vert \sum_{\ell} v_\ell \right\Vert_{1,k}^2  \ = \ & 
\left(\sum_\ell v_\ell ,  \sum_{\ell'} v_{\ell'}\right)_{1,k}  
\  = \  \sum_\ell \sum_{\ell' \in \Lambda(\ell)}  (v_{\ell} ,
v_{\ell'})_{1,k}   
\nonumber \\
 \  \leq \  & \sum_\ell \Vert v_\ell \Vert_{1,k} \left(\sum_{\ell' \in \Lambda(\ell)}  \Vert v_{\ell'}\Vert_{1,k}\right)\nonumber \\
\  \leq \  & \left(\sum_\ell \Vert v_\ell \Vert_{1,k}^2\right)^{1/2}  
\left(\sum_\ell \left(\sum_{\ell' \in \Lambda(\ell)}  
\Vert v_{\ell'}\Vert_{1,k}
\right)^2 \right)^{1/2} \nonumber \\
\  \leq \  & 
\left(\sum_\ell \Vert v_\ell \Vert_{1,k}^2\right)^{1/2}  
\left(\sum_\ell  \# \Lambda(\ell) \sum_{\ell' \in \Lambda(\ell)}  
\Vert v_{\ell'}\Vert_{1,k}^2 
 \right)^{1/2} 
\  \lesssim\     
\sum_\ell \Vert v_\ell \Vert_{1,k}^2 \ , \label{eq:sum_from_1} 
\end{align}
where we used the finite overlap assumption \eqref{eq:finoverlap}. 
To obtain   \eqref{eq:norm_of_sum}, we
write    
\begin{align}
\left\Vert \sum_{\ell=0}^N v_\ell \right\Vert_{1,k}^2  \  = \ & 
\left(\sum_{\ell=0}^N v_\ell ,  \sum_{\ell=0}^N v_\ell\right)_{1,k} \label{eq:rheq}\\
\ = \ & \Vert v_0\Vert_{1,k}^2 \ + \ 2 \left(v_0, \sum_{\ell} v_\ell\right)_{1,k}
\ + \ \left(\sum_\ell v_\ell ,  \sum_\ell v_\ell\right)_{1,k}. \nonumber
\end{align}
Using  the  Cauchy-Schwarz and the  arithmetic-geometric mean
inequalities on the middle term  we can   estimate \eqref{eq:rheq}
from   above in the form  
$$ \lesssim \ \Vert v_0\Vert_{1,k}^2 + \left\Vert \sum_\ell v_\ell
\right\Vert_{1,k}^2 ,  $$
and the result follows from \eqref{eq:sum_from_1}. 
 \end{proof}

\vspace{0.3cm}

Now for each $\ell = 0, \ldots, N$,  we define  linear operators $Q_{\eps, \ell} : 
H^1(\Omega) \rightarrow \cV_{\ell}$ as follows. 
For each $v_h \in \red{H^1(\Omega)}$, $Q_{\eps,\ell} \red{v}$ is defined to be the unique solution of 
the equation 
\beq 
\label{eq:defQi} 
\aeps (Q_{\eps,\ell} \red{v} , w_{h,l}) \ =\  \aeps (\red{v} , w_{h,\ell}),  \quad w_{h, \ell } \in \Vell. 
\eeq 
and  we then define 
$$
Q_{\eps} = \sum_{\ell=0}^N Q_{\eps,\ell}\ . $$ 
The   matrix
representation of $Q_\eps$ corresponds to the action of the preconditioner \eqref{eq:defAS} on the matrix    $A_\eps$ 
(this will be shown in  Theorem
\ref{thm:repQ} below).
In Theorems \ref{thm:boundQ} and \ref{thm:final} below we estimate  
the norm and field of values of 
$Q_{\eps}$, and this yields corresponding estimates for
the norm and field of values of the preconditioned matrix in
Theorems \ref{thm:final1}.
Such projection analysis
is commonplace in domain decomposition; however, as far as we are aware,
this is the first place where the projection operators are defined
using the $a_\eps$ sesquilinear form and analysed in the wavenumber-dependent  $\Vert \cdot
\Vert_{1,k} $ energy norm. 

\begin{theorem} {{(Upper bound on $Q_{\eps}$)}}
\label{thm:boundQ}
\beqs
\nrme{Q_{\eps} v_h} \ \lesssim \ \left(\ksqeps\right) \nrme{v_h}
\quad \text{for all}\quad  v_h \in \cV^h.
\eeqs
\end{theorem}

\begin{proof} 
By the definition of $Q_\eps$ and Lemma \ref{lem:norm_of_sum}, we have 
\begin{equation}\nrme{Q_{\eps}  v_h}^2 \ = \ \left\Vert\sum_{\ell=0}^N
    Q_{\eps,\ell}  v_h\right\Vert_{1,k}^2 \ \lesssim\ 
  \sum_{\ell=0}^N \nrme{Q_{\eps,\ell} v_h}^2\ . \label{eq:641}\end{equation}
Furthermore, by applying Lemma \ref{lem:coer} and the definition 
\eqref{eq:defQi}, we have 
\begin{align*}
\sum_{\ell=0}^N \nrme{\Qepsell v_h}^2 \ \lesssim \ &
\left(\ksqeps\right) \sum_{\ell=0}^N \Im \left(\Theta a_\eps(\Qepsell
v_h,\Qepsell v_h)\right)
\ =  \ 
\left(\ksqeps\right)  \Im \left(\Theta \sum_{\ell=0}^N
  a_{\eps}(v_h,\Qepsell v_h)\right)\ \\
 \ = \ & 
\left(\ksqeps\right)  \Im \left(\Theta  a_{\eps}\left(v_h,\sum_{\ell=0}^N \Qepsell v_h\right)\right)
\ \leq  \ 
\left(\ksqeps\right)  \left\vert  a_{\eps}\left(v_h,\sum_{\ell=0}^N \Qepsell v_h\right)\right\vert\nonumber
\end{align*}
(recalling that  $|\Theta|=1$).
Then, using  Lemma \ref{lem:cont}, and then 
Lemma \ref{lem:norm_of_sum},  
we have 
\begin{align}
\sum_{\ell=0}^N \nrme{\Qepsell v_h}^2 \ \lesssim \ &
\left(\ksqeps\right)  \nrme{v_h} 
\left\Vert \sum_{\ell=0}^N \Qepsell v_h\right\Vert_{1,k}\nonumber \\
\ \lesssim \ &
\left(\ksqeps\right)  \nrme{v_h} \left(\sum_{\ell=0}^N \nrme{\Qepsell v_h}^2\right)^{1/2} \ . \label{eq:642}
 \end{align}
The result follows on combining \eqref{eq:641} with \eqref{eq:642}.
\end{proof}

\begin{remark}  
The  use of the estimate  
$$
 \Im\left( \Theta a_\eps(v,v) \right) \ \gtrsim\  \frac{\vert
   \eps \vert }{k^2}\,  \nrme{v}^2,
$$
which follows from \eqref{eq:coercivity}, is crucial in the proof of Theorem \ref{thm:boundQ}. Indeed, the above proof uses the linearity of the imaginary part of $a(\cdot,\cdot)$ with respect to the 
second argument. The cruder estimate 
$$
\vert a_{\eps}(v,v) \vert \ \gtrsim\ \frac{\vert \eps
  \vert }{k^2}\,  \nrme{v}^2,
$$
which also follows from \eqref{eq:coercivity}, could not be used to
prove Theorem \ref{thm:boundQ}.
\end{remark}

\begin{lemma}\label{lem:lowrestQi}
\begin{align*}
 & \quad \displaystyle{\left( 1 + \frac{H}{\delta} \right)^{1/2}  \left(
  \sum_{\ell=0}^N \Vert Q_{\eps,\ell} v_h \Vert_{1,k}^2 \right)^{1/2} \ \gtrsim \
\frac{\vert \eps \vert }{k^2} \Vert v_h \Vert_{1,k} \ }\quad \tfa v_h \in \cV^h.
\end{align*}
\end{lemma}
\begin{proof}
We first recall  the decomposition of $v_h$ as given in Lemma
\ref{lem:stable_splitting}. 
Then, using Lemma \ref{lem:coer}, the definition of $Q_{\eps,\ell}$ and
Lemma \ref{lem:cont}, 
we obtain:
\begin{align*}
\frac{\vert \eps \vert }{k^2} \Vert v_h \Vert_{1,k}^2 \ \lesssim& \ \Im \left[\Theta a_\eps (v_h, v_h)\right] \ = \ \sum_{l=0}^N \Im \left[\Theta a_\eps (v_h, v_l)\right] \\
 =  & \ \sum_{l=0}^N \Im \left[\Theta a_\eps (\Qepsell v_h, v_l)\right] \ \lesssim\  \sum_{l=0}^N \Vert \Qepsell v_h \Vert_{1,k} \Vert v_l\Vert_{1,k} \ . 
\end{align*}
Then applying the Cauchy-Schwarz inequality and Lemma \ref{lem:stable_splitting} yields
\begin{align*}
\frac{\vert \eps\vert }{k^2} \Vert v_h \Vert_{1,k}^2 \ \lesssim& \ 
\left( \sum_{\ell=0}^N \Vert Q_{\eps,\ell} v_h \Vert_{1,k}^2\right)^{1/2}  \left( \sum_{\ell=0}^N 
\Vert v_\ell\Vert_{1,k}^2 \right)^{1/2} 
 \lesssim     \ \left( 1 + \frac{H}{\delta} \right)^{1/2} \,  \left(
  \sum_{\ell=0}^N \Vert Q_{\eps,\ell} v_h \Vert_{1,k}^2\right)^{1/2}  
\Vert v_h\Vert_{1,k}   . 
\end{align*}
\end{proof}


Our next key result (Lemma \ref{lem:estQepsz} below) is an estimate for the $L^2$-error in the coarse space
projection operator $Q_{\eps,0}$; this is crucially needed to get good estimates for the two-grid preconditioner represented by $Q_\eps$. 
In order to prove this result we need to make an assumption about the approximability on the coarse grid of the solution of the adjoint problem. 

\begin{assumption}[Coarse-grid approximability of the adjoint problem]\label{ass:approx}
If $\phi$ is the solution of the adjoint problem 
\begin{subequations}\label{eq:Esub1}
\begin{align}
-\Delta \phi - (k^2 - \ri \eps ) \phi = f \quad \text{on} \quad \Omega, \label{eq:Helm_with_eps}\\
\pdiff{\phi}{n} -  \ri \overline{\eta} \phi=0\quad \text{on} \quad \Gamma\ \label{eq:homogBC},
\end{align}
with $f\in \LtO$, then
\beq\label{eq:ass1}
\inf_{\phi_0\in \cV_0} \N{\phi-\phi_0}_{1,k}\ \lesssim \ kH \left(\frac{k}{|\eps|}\right) \N{f}_{\LtO}.
\eeq
\end{subequations}
\end{assumption}

\begin{theorem}\label{thm:polygon}
Assumption \ref{ass:approx} holds when $\Omega$ is a 2-d convex polygon, $\eta$ satisfies \eqref{eq:condition}, and the coarse grid is as described in \S\ref{sec:DD} (with, in particular, $H$ denoting the mesh diameter).
\end{theorem}

\bpf
If $\phi$ satisfies \eqref{eq:Esub1} and $\eta$ satisfies \eqref{eq:condition}, then the coercivity estimate \eqref{eq:adjcoercivity} combined with the Lax--Milgram theorem implies that
\beq\label{eq:E1}
\N{\phi}_{1,k} \ \lesssim \ \frac{k}{\vert \eps \vert } \N{f}_{\LtO}.
\eeq
If $\Omega$ is a convex polygon, the regularity results in \cite{Gr:85} can then be used to show that 
\beq\label{eq:E2}
\N{\phi}_{H^2(\Omega)}\  \lesssim \ \frac{k^2}{\vert \eps\vert } \N{f}_{\LtO};
\eeq
see \cite[Lemma 2.12]{GaGrSp:13}. Now, with $\phi_0$ denoting the Scott-Zhang quasi-interpolant on 
the coarse grid, 
we have 
\beq\label{eq:E3i}
\inf_{\phi_0\in \cV_0} \N{\phi-\phi_0}_{1,k}\ \lesssim\ H \N{\phi}_{H^2(\Omega)} + kH \N{\phi}_{H^1(\Omega)}
\eeq
\cite[Theorem 4.1]{ScZh:90}, and the result \eqref{eq:ass1} follows from combining \eqref{eq:E1}, \eqref{eq:E2}, and \eqref{eq:E3i}.
\epf

\bre[Bounds on the adjoint problem]
(i) In the proof of Theorem \ref{thm:polygon}, we obtained the bound \eqref{eq:E1} from coercivity and the Lax--Milgram theorem. This bound can also be obtained from an argument 
involving Green's identity (with the latter giving better estimates in the case of an inhomogeneous boundary condition); see \cite[Remark 2.5]{GaGrSp:13} (but note that the $\eta$ in (2.3b) of that paper should be $\overline{\eta}$).

(ii) The bounds \eqref{eq:E1} and \eqref{eq:E2} are the best currently-available bounds 
on the solution of \eqref{eq:Esub1} for $\eps\gtrsim k$, but they are not 
optimal when $\eps \ll k$ -- see \cite[Theorem 2.9]{GaGrSp:13}. 
\ere

\bre[Establishing Assumption \ref{ass:approx} for more general domains]\label{rem:approx}
(i) $H^2$-regularity of the Laplacian on convex polyhedra with homogeneous Dirichlet boundary conditions is proved in \cite[Corollary 18.18]{Da:88a}. The analogous result for inhomogeneous Neumann boundary conditions could then be used, following \cite[Lemma 2.12]{GaGrSp:13}, to prove that \eqref{eq:E2} (and thus also Assumption \ref{ass:approx}) holds for the solution of \eqref{eq:Esub1} on convex polyhedra with quasi-uniform meshes.

(ii) When $\Omega$ is a bounded, non-convex Lipschitz polyhedron in $\Rea^d$, $d=2,3$, it is natural to use a sequence of locally-refined meshes. In this case we expect Assumption \ref{ass:approx} to hold where $H$ is replaced by $(1/N)^{1/d}$, where $N$ is the dimension of the subspace (so $(1/N)^{1/d}$ is the largest element diameter). The steps to prove this are outlined in \cite[Assumption 3.7, Remark 3.8]{GaGrSp:13}.
\ere

We now use Assumption \ref{ass:approx} to prove the key lemma on 
the approximation power of $\Qepsz$ measured in the $L^2$-norm on the domain.
 
\begin{lemma}\label{lem:estQepsz} (Estimate for $\Qepsz$)
For all $v \in H^1(\Omega)$, 
\begin{align} 
\Vert (I - \Qepsz) v \Vert_{L^2(\Omega)}  &\lesssim 
 kH \left(\frac{k}{|\eps|}\right)    \nrme{(I - \Qepsz) v}\  .
\label{eq:estQ0Omega}
\end{align}
\end{lemma}
\begin{proof}
In the proof, for simplicity,  we write $Q_0$ instead of $\Qepsz$. 
Recall that $Q_0$ is defined by the variational problem
$
a_\eps(Q_0 v,w)= a_\eps(v,w),  \,  \tfa w \in \cV_0,
$
and thus $e_0:= (I-Q_0 )v$ satifies
\beq\label{eq:EGO}
a_\eps(e_0,w)=0 \quad \tfa w\in \cV_0,
\eeq
Let $\phi$ be the 
solution of the adjoint problem 
\begin{align*}
-\Delta \phi - (k^2 - \ri \eps ) \phi &=e_0 \quad \text{on} \quad \Omega, \\
\pdiff{\phi}{n} + \ri \overline{\eta}\phi &= 0 
\quad \text{on} \quad \Gamma.
\end{align*}
Then, for all $w \in H^1(\Omega)$, we have $a_{\eps}(w, \phi) =  
(w, e_0)_{L^2(\Omega)}$. Hence, using \eqref{eq:EGO}, 
we can write
\begin{equation}\label{eq:Q0_1}
\Vert e_0\Vert_{L^2(\Omega)}^2 \ = \  \vert a_{\eps} (e_0,\phi)\vert  \ = \  \vert a_{\eps} (e_0, \phi - \phi_0)\vert 
\end{equation}
for any $\phi_0 \in \cV_0$.
Now, by Assumption \ref{ass:approx}, there exists a $\phi_0 \in \cV_0$ such that 
\beqs
\N{\phi-\phi_0}_{1,k} \lesssim kH \left(\frac{k}{|\eps|}\right)\N{f}_{\LtO}.
\eeqs
Therefore, using this last bound and continuity, we have
\beq\label{eq:third_eps}
\vert a_\eps(e_0,\phi- \phi_0) \vert \lesssim  \Vert e_0\Vert_{1,k} \Vert \phi - \phi_0\Vert_{1,k}
 \lesssim  \Vert e_0\Vert_{1,k}(kH) \left(\frac{k}{\vert \eps\vert }\right)\N{e_0}_{L^2(\Omega)},
\eeq
and combining \eqref{eq:third_eps} and 
\eqref{eq:Q0_1} we obtain \eqref{eq:estQ0Omega}.
\end{proof}

In what follows, we need both the Poincar\'{e}--Friedrichs 
inequality and the trace inequality on  domains $D$ of {\em characteristic 
length scale} $L$. By this we mean that $D$ is assumed to have  diameter $\sim L$, 
surface area $\sim L^{d-1}$ and volume $\sim L^d$.   
The estimates in the next two results are then explicit in $L$ (with the hidden constants independent of $L$).    

\begin{theorem}\label{thm:Poin}
If $D$ is a Lipschitz domain with characteristic length scale $L$,  
then the Poincar\'e-Friedrichs inequality is 
\begin{equation}\label{eq:Poin} 
\Vert v \Vert_{L^2(D)} \ \lesssim \ L \vert v \vert_{H^1(D)},   
\end{equation} 
for all $v\in H^1(D)$ that vanish on a subset of 
$\partial D $ with measure $\sim L^{d-1}$,
and the multiplicative trace inequality is
\begin{equation}\label{eq:Trace} 
\Vert {v} \Vert _{L^2({\partial D })}^2 \ \lesssim \
\left( L^{-1} \Vert {v} \Vert_{L^2(D)} + \vert {v} \vert_{H^1({D})} 
\right) \Vert {v} \Vert_{L^2(D)} \ , \quad \text{for all} \quad  v\in H^1(D) \ .
\end{equation} 
\end{theorem}

\begin{proof}
For domains of size $\cO(1)$, \eqref{eq:Poin}  is proved in, e.g., \cite[Theorem 1.9]{Ne:67}, and \eqref{eq:Trace} is proved in 
\cite[Last equation on p. 41]{Gr:85}. A scaling argument then yields \eqref{eq:Poin} and \eqref{eq:Trace}.
\end{proof}

Combining \eqref{eq:Poin} and \eqref{eq:Trace} we obtain the following corollary.

\begin{corollary}
If $D$ is a Lipschitz domain with characteristic length scale $L$ and $v$ vanishes on a subset of $\partial D $  of measure $\sim L^{d-1}$, then 
\begin{equation}\label{eq:Poin_and_Trace} 
\Vert {v} \Vert _{L^2(\partial D)}\ \lesssim \ L^{1/2} |u|_{H^1(D)}\ . 
\eeq
\end{corollary}
At various places in this paper we make use of the simple ``Cauchy inequality'': 
\beq\label{eq:Cauchy}
2ab \leq \delta a^2 + \frac{b^2}{\delta}, \quad a,b,\delta>0.
\eeq  
In particular, using this (with $\delta = 1$)  and   
the multiplicative trace inequality \eqref{eq:Trace}, we obtain another corollary to Theorem \ref{thm:Poin}.
\begin{corollary}
\label{cor:tracek} If $D$ is a Lipschitz domain (with characteristic length scale $\cO(1)$) then
\beq\label{eq:mult_trace}
k^{1/2}\Vert {v} \Vert_{L^2(\partial D)}\ \lesssim \ \N{v}_{1,k}\ , \quad \text{for all}  \quad v\in H^1(D) \quad \text{and}  \quad  k \geq 1 .
\eeq   
\end{corollary}   
  
Our goal for the rest of the section is to bound the field of values 
$ (v_h, Q_\eps v_h)_{1,k}/ \|v_h\|_{1,k}^2$ away from the origin in the complex plane. 
(Note that the field of values is computed with respect 
to the $(\cdot, \cdot)_{1,k} $ inner product.)  We do this by 
 estimating  $\vert(v_h, Q_\eps v_h)_{1,k}\vert$ below by $\sum_{l=0}^N\|Q_{\eps,l}v_h\|^2_{1,k}$ plus ``remainder" terms (which turn out to be higher order, i.e.~bounded by a positive power of $H$ or $\Hsub$), and then use Lemma \ref{lem:lowrestQi} to bound the sum below by $\N{v_h}_{1,k}^2$. Lemma \ref{lem:combined} 
sets up the ``remainder" terms, $R_{\eps,\ell}(v_h)$, Lemmas \ref{lem:coarse} and \ref{lem:local} estimate these, and the final result is then given in Theorem \ref{thm:final}. 

\begin{lemma}\label{lem:combined}  For $\ell = 0, \ldots, N$, set
\beq\label{eq:E3}
R_{\eps,\ell} (v_h) \ : = \ \left((I - Q_{\eps, \ell})v_h, Q_{\eps,\ell} v_h \right)_{1,k}.
\eeq
Then
\begin{align}
 (v_h, Q_\eps v_h )_{1,k} 
=  \sum_{\ell = 0}^N
\left\{ \Vert Q_{\eps, \ell} v_h \Vert_{1,k}^2 + R_{\eps,\ell} (v_h) \right\}. \label{eq:E4ii}
\end{align}
Furthermore, $R_{\eps,\ell}$ satisfies 
\beq\label{eq:boundE}
\vert R_{\eps,\ell}(v_h) \vert   \ \lesssim \ D_{\eps,\ell}(v_h)  +  B_{\eps,\ell}(v_h), 
\eeq
where  the ``domain'' and ``boundary'' contributions to the bound are given by  
\begin{align}\label{eq:Ddef}
D_{\eps,\ell}(v_h) &=  k^2 \Vert (I- Q_{\eps, \ell})v_h\Vert_{L^2(\Omega_\ell)} \Vert Q_{\eps, \ell}v_h\Vert_{L^2(\Omega_\ell)},\\
 B_{\eps,\ell}(v_h) & =  k \Vert (I- Q_{\eps, \ell})v_h\Vert_{L^2(\Gamma_\ell)} \Vert Q_{\eps, \ell}v_h\Vert_{L^2(\Gamma _\ell)}.\label{eq:Bdef}
\end{align}
and $\Omega_0 = \Omega$, $\Gamma_0 = \Gamma$, and $\Gamma_\ell = \Gamma\cap \partial \Omega_\ell$, for $\ell = 1, \ldots, N$.
\end{lemma} 

\begin{proof}
By the definition of $Q_\eps$,
\beqs
(v_h, Q_\eps v_h )_{1,k} = \sum_{\ell = 0}^N 
\left(v_h, Q_{\eps, \ell} v_h \right)_{1,k}= \sum_{\ell = 0}^N
\left\{ \Vert Q_{\eps, \ell} v_h \Vert_{1,k}^2 + \left((I - Q_{\eps, \ell})v_h, Q_{\eps,\ell} v_h \right)_{1,k} \right\} \ ,  
\eeqs
yielding  \eqref{eq:E4ii}. To obtain \eqref{eq:boundE}, we recall that.
\beqs
(u,v)_{1,k} = a_\eps(u,v) + (2k^2 + \ri \eps) (u,v)_{L^2(\Omega)} + \ri \eta (u,v)_{\LtG}.
\eeqs
Then, since
\beqs
a_\eps((I - Q_{\eps, \ell})v_h, Q_{\eps,\ell} v_h)=0
\eeqs
(from the definition of $Q_{\eps,\ell}$ \eqref{eq:defQi}),
we have 
\beqs
R_{\eps,\ell} (v_h) =(2k^2 + \ri \eps)  ((I - Q_{\eps, \ell})v_h, Q_{\eps,\ell} v_h)_{L^2(\Omega_\ell)}
+ \ri \eta  ((I - Q_{\eps, \ell})v_h, Q_{\eps,\ell} v_h)_{L^2(\Gamma_\ell)},
\eeqs
where we have also used the fact that $Q_{\eps, \ell}v_h$ has support only on $\Omega_l$. 
The desired ``domain" and ``boundary" estimates \eqref{eq:Ddef} and \eqref{eq:Bdef} then follow after using \eqref{eq:est_epseta}. 
\end{proof}

We now bound $D_{\eps,\ell}(v_h)$ and $B_{\eps,\ell}(v_h)$, using the following strategy. 
 First Lemma \ref{lem:coarse}  
bounds  $D_{\eps,0}(v_h)$ in terms of  a positive power of $H$, which  
is obtained by using Lemma \ref{lem:estQepsz} to  estimate the $\|(I-Q_{\eps,0})v_h\|_{L^2(\Omega)}$ component of $D_{\eps,0}(v_h)$.
Then,        
in Lemma \ref{lem:local}  we 
bound  $\sum_{\ell =1}^N D_{\eps,\ell}(v_h)$ in terms of  a positive power of $\Hsub$,  by 
applying  the Poincar\'e--Friedrichs inequality \eqref{eq:Poin}
to each of  the  $\|Q_{\eps,\ell}v_h\|_{L^2(\Omega_\ell)}$ terms in this sum. 
These two lemmas also provide  bounds on $B_{\eps,0}(v_h)$ and $\sum_{\ell =1}^N B_{\eps,\ell}(v_h)$, respectively, where similar  ideas  are used, except this time in conjunction with trace inequalities.  Recalling  that $H$ is the coarse mesh diameter and  $\Hsub$ the subdomain diameter, we are then able  to control the error terms by making $H$ and $ \Hsub$ sufficiently small    (Theorem \ref{thm:final}); it turns out that the required condition on $H$ is more stringent than that on $\Hsub$.

\begin{lemma}[Bounds on $D_{\eps,0}$ and $B_{\eps,0}$]\label{lem:coarse}
For any $\alpha,\alpha'\geq 0$ and any $v_h\in \cV^h$,
\beq\label{eq:E5}
D_{\eps,0}(v_h) \lesssim k H \kseb \left[ \kseb^\alpha \N{Q_{\eps,0}v_h}^2_{1,k} + \kseb^{-\alpha}\N{v_h}^2_{1,k}\right],
\eeq
\beq\label{eq:E5a}
B_{\eps,0}(v_h) \lesssim (kH)^{1/2} \kseb^{1/2} \left[ \kseb^{\alpha'} \N{Q_{\eps,0}v_h}^2_{1,k} + \kseb^{-\alpha'}\N{v_h}^2_{1,k}\right],
\eeq
and thus (taking $\alpha'=\alpha$) 
\begin{align}\nonumber
&D_{\eps,0}(v_h)+B_{\eps,0}(v_h) \\
&\lesssim  (kH)^{1/2} \kseb^{1/2}\left[1+  (kH)^{1/2} \kseb^{1/2}\right] \left[ \kseb^{\alpha} \N{Q_{\eps,0}v_h}^2_{1,k} + \kseb^{-\alpha}\N{v_h}^2_{1,k}\right].\label{eq:E10}
\end{align}
\end{lemma}

\bpf
For the bound on $D_{\eps,0}(v_h)$, we use \eqref{eq:estQ0Omega}, the triangle inequality, and the Cauchy inequality \eqref{eq:Cauchy} to obtain
\begin{align}\nonumber
D_{\eps,0}(v_h) &\ \lesssim \  kH \kseb\N{(I-Q_{\eps,0})v_h}_{1,k} \N{Q_{\eps,0}v_h}_{1,k},\\
&\lesssim \ kH \kseb \left[
\N{Q_{\eps,0}v_h}^2_{1,k} + \N{Q_{\eps,0}v_h}_{1,k}\N{v_h}_{1,k}\right]\label{eq:E6},\\
&\lesssim \ kH \kseb \left[ \N{Q_{\eps,0}v_h}^2_{1,k} + \kseb^{\alpha} \N{Q_{\eps,0}v_h}_{1,k}^2 + \kseb^{-\alpha}\N{v_h}^2_{1,k}\right]\ , 
\end{align}
for any $\alpha \geq 0 $. Since  $\veps\lesssim k^2$,  \eqref{eq:E5} follows.  

For the bound on $B_{\eps,0}(v_h)$, we apply  the multiplicative trace inequality \eqref{eq:Trace} on $\Omega$  (so $L\sim 1$),  to obtain 
\beqs
B_{\eps,0}(v_h) \lesssim k \N{(I-Q_{\eps,0})v_h}_{\LtO}^{1/2} \N{(I-Q_{\eps,0})v_h}_{\HoO}^{1/2} \N{Q_{\eps,0}v_h}_{\LtG}.
\eeqs
Using \eqref{eq:estQ0Omega}, we then have 
\beqs
B_{\eps,0}(v_h) \lesssim k\, (kH)^{1/2} \left(\frac{k}{\veps}\right)^{1/2}\N{(I-Q_{\eps,0})v_h}_{1,k}^{1/2} \N{(I-Q_{\eps,0})v_h}_{\HoO}^{1/2} \N{Q_{\eps,0}v_h}_{\LtG}
\eeqs
and then using \eqref{eq:mult_trace} and the triangle inequality we obtain
\begin{align*}
B_{\eps,0}(v_h)& \lesssim (kH)^{1/2} \kseb^{1/2}\N{(I-Q_{\eps,0})v_h}_{1,k}\N{Q_{\eps,0}v_h}_{1,k}\\
 &\lesssim (kH)^{1/2} \kseb^{1/2}\left[
 \N{Q_{\eps,0}v_h}^2_{1,k} +  \N{v_h}_{1,k}\N{Q_{\eps,0}v_h}_{1,k}
\right].
\end{align*}
This last inequality is the analogue of \eqref{eq:E6}, and proceeding as before we obtain \eqref{eq:E5a}. 
\epf

\begin{lemma}[Bounds on $\sum D_{\eps,\ell}$, $\sum B_{\eps,\ell}$]\label{lem:local}
For any $\alpha, \alpha'\geq 0$ and any $v_h\in \cV^h$,
\beq\label{eq:E7}
\sum_{\ell= 1}^N D_{\eps,\ell} (v_h)\ \lesssim \ k \Hsub  \left[
\kseb^\alpha \sum_{\ell=1}^N \N{Q_{\eps,l}v_h}^2_{1,k} + \kseb^{-\alpha} \N{v_h}^2_{1,k}\right]
\eeq
and
\beq\label{eq:E8}
\sum_{\ell= 1}^N B_{\eps,\ell} (v_h)\ \lesssim \ k \Hsub  \left[
\kseb^{\alpha'} \sum_{\ell=1}^N \N{Q_{\eps,l}v_h}^2_{1,k} + \kseb^{-\alpha'} \N{v_h}^2_{1,k}\right].
\eeq
Therefore (letting $\alpha'=\alpha$),
\beq\label{eq:E11}
\sum_{\ell= 1}^N \big[ D_{\eps,\ell} (v_h)+ B_{\eps,\ell} (v_h)\big] 
\ \lesssim \ k \Hsub 
\left[
\kseb^{\alpha} \sum_{\ell=1}^N \N{Q_{\eps,l}v_h}^2_{1,k} + \kseb^{-\alpha} \N{v_h}^2_{1,k}\right].
\eeq
\end{lemma}
\begin{proof} Let $\ell = 1, \ldots, N$. 
Recalling both that $Q_{\eps,\ell} v_h $ vanishes on 
$\partial \Omega_\ell \backslash \Gamma$ and the assumption \eqref{eq:int_bound}, we can use the Poincar\'{e} inequality \eqref{eq:Poin} on $\Omega_\ell$, and then use the triangle inequality to obtain
\begin{align*}
D_{\eps,\ell}(v_h)&\ \lesssim \  k^2\Hsub  \N{(I-Q_{\eps,\ell})v_h}_{L^2(\Omega_\ell)} |Q_{\eps,\ell}v_h|_{H^1(\Omega_l)},\\
&\ \lesssim\  k \Hsub  \left[ \N{Q_{\eps,\ell}v_h}_{1,k}^2  + k\N{v_h}_{L^2(\Omega_\ell)} \N{Q_{\eps, \ell}v_h}_{1,k}\right]\ 
\end{align*}
(where the $1,k$-norm is over  the support of 
$Q_{\eps,\ell}v_h$, which is $\Omega_\ell$).
Using  \eqref{eq:Cauchy}
we obtain
\beqs
D_{\eps,\ell}(v_h)\lesssim k \Hsub  \left[ \kseb^{\alpha} \N{Q_{\eps,\ell}v_h}^2_{1,k} + k^2 \kseb^{-\alpha} \N{v_h}^2_{L^2(\Omega_\ell)}\right], 
\eeqs
with $\alpha\geq0$. Summing from $\ell=1$ to $N$, 
and using the finite-overlap property \eqref{eq:finoverlap},  gives \eqref{eq:E7}.

From 
\eqref{eq:Bdef} we have 
\beqs
B_{\eps,\ell}(v_h) \ \lesssim\  k \left[\N{Q_{\eps,\ell}v_h}_{L^2(\Gamma_\ell)}^2+ \N{v_h}_{L^2(\Gamma_\ell)} \N{ Q_{\eps,\ell} v_h}_{L^2(\Gamma_\ell)}\right]
\eeqs
and then using  \eqref{eq:Poin_and_Trace} we have
\beqs
B_{\eps,\ell}(v_h)\ \lesssim \ k \left[ \Hsub |Q_{\eps,\ell}v_h|^2_{H^1(\Omega_\ell)} + \Hsub^{1/2}\N{v}_{L^2(\Gamma_\ell)}|Q_{\eps,\ell}v_h|_{H^1(\Omega_\ell)}\right].
\eeqs
Summing from $\ell=1$ to $N$ we then obtain
\beq\label{eq:E9}
\sum_{\ell=1}^N B_{\eps,\ell}(v_h)\ \lesssim\  k\Hsub \sum_{\ell=1}^N |Q_{\eps,\ell}v_h|^2_{H^1(\Omega_\ell)} + k \Hsub^{1/2}\sum_{\ell=1}^N \N{v_h}_{L^2(\Gamma_\ell)}|Q_{\eps,\ell}v_h|_{H^1(\Omega_\ell)}.
\eeq
(Note that the sums in the last inequality  could  be restricted to those  $\ell$ with 
$\Gamma \cap \partial \Omega_\ell \not = \emptyset$, but this is not used in the following.) 
Using the Cauchy-Schwarz inequality, then  \eqref{eq:Poin_and_Trace} and finally 
 \eqref{eq:Cauchy}, we have 
\begin{align}
k \Hsub^{1/2} \sum_{\ell=1}^N \N{v_h}_{L^2(\Gamma_l)} |Q_{\eps,\ell}v_h|_{H^1(\Omega_\ell)}
&\ \lesssim \ k \Hsub^{1/2} \left( 
\sum_{\ell=1}^N \N{v_h}^2_{L^2(\Gamma_\ell)}\right)^{1/2} \left( 
\sum_{\ell=1}^N |Q_{\eps,\ell}v_h|^2_{H^1(\Omega_\ell)}\right)^{1/2}, \nonumber \\
&\ \lesssim \ k \Hsub \N{v_h}_{1,k}\left( 
\sum_{\ell=1}^N \N{Q_{\eps,\ell}v_h}^2_{1,k}\right)^{1/2}, \nonumber \\
&\ \lesssim \ k \Hsub \left[
\kseb^{\alpha'} \sum_{\ell=1}^N \N{Q_{\eps,\ell}v_h}_{1,k}^2 + \kseb^{-\alpha'} \N{v_h}^2_{1,k}\right].
\label{eq:E9a}
\end{align}
Inserting \eqref{eq:E9a} into \eqref{eq:E9}, we  obtain the result \eqref{eq:E8}.
\end{proof}

\

Our main result in the rest of this section is the following estimate from below on the field of values of $Q_\eps$. 

\begin{theorem}[Bound below on the field of values] \label{thm:final} 
There exists a  constant $\cC_1>0$ 
such that
\beq\label{eq:E12}
\vert (v_h, Q_\eps v_h )_{1,k} \vert \ \gtrsim \ \left(1+ \frac{H}{\delta}\right)^{-1}\left(\frac{\vert \eps\vert }{k^2} \right)^{2}\ \Vert v_h \Vert_{1,k}^2 \ , \quad \text{for all} \quad  v_h \in \cV^h,
\eeq
when 
\beq\label{eq:E20}
\max\left\{ k\Hsub,\  kH \left(1+ \frac{H}{\delta}\right)  \kseb^2 \right\}  \ \leq \  
\cC_1 \left(1 + \frac{H}{\delta}\right)^{-1}  \left(\frac{\vert \eps \vert }{k^2}\right). 
\eeq
\end{theorem}
Note that the condition on the coarse mesh diameter $\Hsub$ is more stringent than the condition on the subdomain diameter $H$; one finds similar criteria in domain-decomposition theory for coercive elliptic PDEs; see, e.g., \cite{GrLeSc:07}.

The following corollary 
restricts attention to 
 a commonly encountered situation. 
\begin{corollary} Suppose $\delta \sim \Hsub \sim H$. There exists a  constant $\cC_1>0$ 
such that
\beq\label{eq:E121}
\vert (v_h, Q_\eps v_h )_{1,k} \vert \ \gtrsim \ \left(\frac{\vert \eps\vert }{k^2} \right)^{2}\ \Vert v_h \Vert_{1,k}^2 \ , \quad \text{for all} \quad  v_h \in \cV^h, 
\eeq
when
\beqs
  kH     \ \leq \   \cC_1 \left(\frac{\vert \eps \vert }{k^2}\right)^3.
\eeqs
\end{corollary}
 
\bpf[Proof of Theorem \ref{thm:final}]
By Lemma \ref{lem:combined},
\beqs
|(v_h, Q_\eps v_h)_{1,k}| \ \gtrsim\  \sum_{\ell=0}^N \N{Q_{\eps,\ell}v_h}_{1,k}^2 - \sum_{\ell=0}^N \left( D_{\eps,\ell}(v_h) + B_{\eps,\ell}(v_h)\right) .
\eeqs
Then, using the bounds \eqref{eq:E10} and \eqref{eq:E11} 
we have
\begin{align*}
&|(v_h, Q_\eps v_h)_{1,k}|\  \gtrsim \ \sum_{\ell=0}^N \N{Q_{\eps,\ell}v_h}_{1,k}^2 \\
&\qquad- (kH)^{1/2} \kseb^{1/2} \left( 1 + (kH)^{1/2} \kseb^{1/2}\right) \left[ \kseb^{\alpha} \N{Q_{\eps,0}v_h}_{1,k}^2 + \kseb^{-\alpha} \N{v_h}^2_{1,k}\right] \\
&\qquad-(k\Hsub)  \left[ \kseb^{\alpha'} \sum_{\ell=1}^N\N{Q_{\eps,\ell}v_h}_{1,k}^2 + \kseb^{-\alpha'} \N{v_h}^2_{1,k}\right] 
\end{align*}
for $\alpha,\alpha'\geq 0$.
Therefore, there exist $C_1, C_2>0$ (sufficiently small) such that
\beq\label{eq:N1}
(kH)^{1/2} \kseb^{(1+2\alpha)/2} \left( 1 + (kH)^{1/2}\kseb^{1/2}\right) \ \leq \ C_1,
\eeq
and
\beq\label{eq:N2}
(k\Hsub) \kseb^{\alpha'}\ \leq \ C_2
\eeq
ensure that
\begin{align}
\nonumber
|(v_h, Q_\eps v_h)_{1,k}| \ \gtrsim \ & \sum_{\ell=0}^N \N{Q_{\eps,\ell}v_h}_{1,k}^2
- (k\Hsub) \kseb^{-\alpha'} \N{v_h}^2_{1,k}\\
&-(kH)^{1/2} \kseb^{1/2} \left( 1 + (kH)^{1/2} \kseb^{1/2}\right)\kseb^{-\alpha} \N{v_h}^2_{1,k}.
\label{eq:N3}
\end{align}
Since $\alpha\geq0$, there exists a $\widetilde{C}_1>0$ such that
\beq\label{eq:N4}
(kH)^{1/2} \kseb^{(1+2\alpha)/2}\ \leq \ \widetilde{C}_1
\eeq
ensures that \eqref{eq:N1} holds; i.e.~\eqref{eq:N3} holds under \eqref{eq:N4} and \eqref{eq:N2}.

Using in \eqref{eq:N3} the bound in Lemma \ref{lem:lowrestQi}, we obtain 
\begin{align*}|(v_h, Q_\eps v_h)_{1,k}| \ \gtrsim \ &\left(1+ \frac{H}{\delta}\right)^{-1} \eksb^2 \N{v_h}^2_{1,k} -(k\Hsub) \kseb^{-\alpha'} \N{v_h}^2_{1,k}\\
&-(kH)^{1/2} \kseb^{1/2} \left( 1 + (kH)^{1/2} \kseb^{1/2}\right)\kseb^{-\alpha} \N{v_h}^2_{1,k}.
\end{align*}
Therefore, there exist $C_3, C_4>0$ (sufficiently small) so that the conditions
\beq\label{eq:E13}
(kH)^{1/2} \kseb^{(1-2\alpha)/2} \left( 1 + (kH)^{1/2} \kseb^{1/2}\right) \ \leq \ C_3 \left(1+ \frac{H}{\delta}\right)^{-1} \eksb^2
\eeq
and
\beq\label{eq:E13a}
(k\Hsub) \kseb^{-\alpha'}\ \leq \ C_4  \left(1+ \frac{H}{\delta}\right)^{-1} \eksb^2,
\eeq
together with \eqref{eq:N3} and \eqref{eq:N4}, ensure that the result \eqref{eq:E12} holds.

Now, condition \eqref{eq:E13} can be rewritten as 
\beqs
(kH)^{1/2} \kseb^{1/2} \left( 1 + (kH)^{1/2} \kseb^{1/2}\right) \ \leq \ C_3 \left(1+ \frac{H}{\delta}\right)^{-1} \eksb^{2-\alpha}.
\eeqs
Now, since $\eps\lesssim k^2$, 
\beqs
 \left(1+ \frac{H}{\delta}\right)^{-1} \eksb^\beta\lesssim 1
\eeqs
for any  $\beta\geq 0$. Therefore, if $\alpha\leq 2$, then there exists a $\widetilde{C_3}>0$ such that the condition \eqref{eq:E13} is ensured by the condition 
\beqs
(kH)^{1/2} \kseb^{1/2}\leq \widetilde{C}_3 \left(1+ \frac{H}{\delta}\right)^{-1} \eksb^{2-\alpha}.
\eeqs
i.e.
\beq\label{eq:E15}
(kH)^{1/2} \kseb^{(5-2\alpha)/2} \leq \widetilde{C_3} \left(1+ \frac{H}{\delta}\right)^{-1} \ .
\eeq

In summary (from \eqref{eq:N4}, \eqref{eq:N2}, \eqref{eq:E15}, and \eqref{eq:E13a}) we have shown that there exist $\widetilde{C}_1, C_2, \widetilde{C}_3, C_4>0$ such that the required 
result \eqref{eq:E12}, holds if the following four conditions hold:
\beq\label{eq:E16}
(kH)^{1/2} \kseb^{(1+2\alpha)/2} \ \leq \ \widetilde{C}_1,
\eeq
\beq\label{eq:E17}
(k\Hsub) \kseb^{\alpha'}\ \leq \ C_2,
\eeq
\beq\label{eq:E18}
(kH)^{1/2} \kseb^{(5-2\alpha)/2} \left(1 + \frac{H}{\delta}\right) \leq\widetilde{C}_3,
\eeq
and
\beq\label{eq:E19}
(k\Hsub) \kseb^{2-\alpha'}\left(1 + \frac{H}{\delta}\right)\leq C_4,
\eeq
where $0\leq \alpha\leq 2$ and $\alpha'\geq 0$. 

The optimal choice of $\alpha$ to balance the exponents in \eqref{eq:E16} and \eqref{eq:E18} (ignoring the factor $(1+ H/\delta)$) is $\alpha=1$, and the optimal choice of $\alpha'$ to balance the exponents in \eqref{eq:E17} and \eqref{eq:E19} (again ignoring $(1+ H/\delta)$) is $\alpha'=1$. With these values of $\alpha$ and $\alpha'$, the four conditions above are ensured by the  condition \eqref{eq:E20}. 
\epf  
 
\bre[One-level methods]
Inspecting the proof of Theorem \ref{thm:final}, we see that the bound from below on the field of values relies on the bound in Lemma \ref{lem:lowrestQi}, which in turn relies on the second bound in Lemma \ref{lem:stable_splitting}. In the case of the one-level method (i.e.~$A_\eps$ is preconditioned with \eqref{eq:minor}), the constant on the right-hand side of the analogue of the second bound in \eqref{eq:stable_splitting} does not $\sim 1$ when $\delta \sim H$; instead it blows up as $H\tendo$. This is why we do not currently have a result analogous to Theorem \ref{thm:final} for the one-level method.
\ere

\section{Matrices and convergence of GMRES}
\label{sec:Matrices}

In this section we interpret the results of Theorems \ref{thm:boundQ} and  \ref{thm:final} in 
terms of  matrices and explain their   implications  for the convergence of GMRES for the Helmholtz equation. 
Let us begin by recalling the  convergence theory for GMRES due originally to Elman \cite{El:82}  
and Eisenstadt, Elman and Schultz \cite{EiElSc:83}, and used in the context of domain decomposition methods in   \cite{CaWi:92}.  
The most convenient statement for our purposes is   
\cite{BeGoTy:06}. 
We consider any abstract  linear system 
 \begin{equation} \label{eq:abstract}
\matrixC \bfx = \bfd \end{equation}
in $\mathbb{C}^n$, where $\matrixC$ is an  $n\times n$ nonsingular complex matrix.   
Choose an initial guess $\bfx^0$ , introduce the residual $\bfr^0 = \bfd- C \bfx^0$ and 
the usual Krylov spaces:  
$$  \mathcal{K}^m(C, \bfr^0) := \mathrm{span}\{\matrixC^j \bfr^0 : j = 0, \ldots, m-1\} \ .$$

Let $\langle \cdot , \cdot \rangle_D$ denote the inner product on $\C^n$ 
induced by some Hermitian positive definite  matrix $D$, i.e.  
\begin{equation}\label{eq:wip}
\langle \bV, \bW\rangle_D := \bW^*D\bV
\end{equation} 
with induced norm $\Vert \cdot \Vert_D$, where $^*$ denotes Hermitian transpose. For $m \geq 1$, define   $\bfx^m$  to be  the unique element of $\cK^m$ satisfying  the   
 minimal residual  property: 
$$ \ \Vert \bfr^m \Vert_D := \Vert \bfd - \matrixC \bfx^m \Vert_D \ = \ \min_{\bfx \in \mathcal{K}^m(C, \br^0)} \Vert {\bfd} - {\matrixC} {\bfx} \Vert_D  , $$
When $D = I$ this is just the usual GMRES algorithm, and we write $\Vert \cdot \Vert = \Vert \cdot \Vert_I$, but for  more general  $D$ it 
is the weighted GMRES method \cite{Es:98} in which case  
its implementation requires the application of the weighted Arnoldi process \cite{GuPe:14}. In  \S \ref{sec:Numerical} we give results for standard  GMRES which corresponds to the case $D= I$ and also for a weighted variant with respect to a certain matrix $D$ defined by \eqref{eq:weight} below.  
The following theorem  is a simple generalisation of the classical convergence result 
stated in  \cite{BeGoTy:06}.

\begin{theorem} \label{Thm:A1}  Suppose    $0 \not \in W_D(\matrixC)$.  \  Then
\begin{equation}\label{eq:GMRESest}
\frac{\Vert \bfr^m \Vert_D} { \Vert \bfr^0 \Vert_D} \ \leq  \ \sin^m (\beta)   \ , 
\quad \text{where} \quad  \cos(\beta) : =
\frac{\mathrm{dist}(0,W_D(\matrixC))}{\Vert \matrixC\Vert_D }\ ,  \end{equation}
where $W_D(C)$ denotes the  {\em field of values} (also called the {\em numerical range} of $C$) with respect to the inner product induced by $D$, i.e.  
$$W_D(C) \ = \ \{\langle \bfx,\matrixC \bfx\rangle_D: \bfx \in \mathbb{C}^n, 
\Vert \bfx \Vert_D = 1\}. $$ 
\end{theorem}

\begin{proof}
For the ``standard'' case $D=I$ the result is stated in \cite{BeGoTy:06}. 
For general  $D$,  
write $\widetilde{C} = D^{1/2}CD^{-1/2}$, 
$\widetilde{\bfd} = D^{1/2}\bfd$, $\widetilde{\bfx} = D^{1/2}\bfx$, $\widetilde{\bfx}^m = D^{1/2}\bfx^m$, and $\widetilde{\bfr^0}= D^{1/2}\bfr_0$. Then it is 
easy to see that $\widetilde{x}^m \in \cK(\widetilde{C}, \widetilde{r}^0)$ and it  satisfies the  ``standard'' 
GMRES criterion for the transformed system but in the Euclidean norm: 
$$ \Vert \widetilde{\bfr}^m \Vert : =  \Vert \widetilde{\bfd} - \widetilde{\matrixC} \widetilde{\bfx}^m \Vert \ = \ \min_{\widetilde{\bfx} \in \mathcal{K}^m(\widetilde{C}, \widetilde{\bfr^0})} \Vert \widetilde{\bfd} - \widetilde{\matrixC} \widetilde{\bfx}\Vert  . $$ 
Then we know that the result 
\eqref{eq:GMRESest}  holds with $D = I$, $\matrixC= \widetilde{\matrixC}$ and  $\bfr^m = \widetilde{\bfr}^m$. 
It is then simple to transform this back to obtain  \eqref{eq:GMRESest} in the case of general $D$. 
\end{proof} 

\begin{remark}
Note that for all $\bfx \in \C^n$ with $\Vert \bfx \Vert_D = 1$, 
we have 
$$ 0 \ \leq \ \mathrm{dist}(0,W_D(C))\  \leq \ \vert \langle   \bfx, \matrixC \bfx\rangle_D\vert \leq \Vert C \Vert_D$$
and so the second formula in \eqref{eq:GMRESest}  necessarily defines an angle $\beta$ 
in the range $[0, \pi/2]$. 
Thus, for good GMRES convergence we aim to ensure that   
$\mathrm{dist}(0,W_D(\matrixC))$ is bounded well away from zero and that  
$\Vert \matrixC\Vert_D $ is as small as possible. 
Theorem \ref{Thm:A1} could therefore be viewed as a generalisation  to the  case of GMRES of the familiar 
condition number criterion for the convergence of the conjugate gradient  
method for positive definite  systems.
 The result of Theorem \ref{Thm:A1} is stated 
without proof in \cite{CaWi:92},  with a reference to \cite{El:82}; however \cite{El:82} is concerned only with standard GMRES in the Euclidean inner product.
\end{remark}

\begin{remark}\label{rem:energy}
As we see in Theorems  \ref{thm:final1} and \ref{thm:final3} below,   
the analysis of \S \ref{sec:Convergence}  provides us with estimates for the norm and field of values of the preconditioned matrix in the { weighted norm} induced by the real symmetric positive matrix $D_k$ defined in \eqref{eq:weight} below. 
Other analyses of domain decomposition methods for 
non self-adjoint or  non-positive definite PDEs (e.g. \cite{CaZo:02}, \cite{SaSz:07}) have 
arrived at analogous estimates in weighted norms, 
although the weights appearing  in these  
previous analyses are different, being  
associated with either the standard $H^1$ norm or semi-norm,  
and not the $k$- weighted energy norm, appropriate for Helmholtz problems, used here.     
\end{remark}

We now use the theory in \S \ref{sec:Convergence}  to obtain results about  
the iterative solution of the linear systems arising from the  
Helmholtz equation. 
We start by  interpreting   the operators $Q_{\eps,\ell}$ defined in \eqref{eq:defQi} 
in terms of matrices. 

\begin{theorem} \label{thm:repQ}
Let $v_h = \sum_{j\in \cI^h} V_j \phi_j \ \in \ \cV^h$. Then 
\begin{align*}
(i)  \quad Q_{\eps,\ell} v_h \ & = \sum_{j \in \cI^h({\Omega_\ell})} \left(R_\ell^TA_{\eps,\ell}^{-1} R_\ell A_\eps \mathbf{V}\right)_j \phi_j \ ,   \quad \ell = 1, \ldots, N\ , \\
(ii)  \quad Q_{\eps,0} v_h \ & = \sum_{p \in \cI^H} \left(R_0^TA_{\eps,0}^{-1} R_0 A_\eps 
\mathbf{V}\right)_p \Phi_p   \ ,  
\end{align*} 
with $A_{\eps,\ell}\ , \ell = 0, \ldots, N$ defined in \eqref{eq:minor} and \eqref{eq:coarsegrid}. 
\end{theorem}
\begin{proof}
These results are similar to those for symmetric elliptic problems found for example in \cite{ToWi:05}, so we will be brief. 
For (i), let $\ell \in \{1, \ldots, N\}$, let  $w_{h,\ell}$ and $y_{h,\ell}$ be  arbitrary elements of $\cV_\ell$, and denote  their  coefficient vectors  
$\mathbf{W}$  and $\mathbf{Y}$ (with nodal values on all of $\cI^h$).  
Then   $\mathbf{W} = R_\ell^T\mathbf{w}$  and $\mathbf{Y}= R_\ell^T\mathbf{y}$, where $\mathbf{w}, \mathbf{y}$ have nodal values on $\cI^h(\Omega_\ell)$. 
The definitions of $A_\eps$ and $A_{\eps,\ell}$, \eqref{eq:discrete} and \eqref{eq:minor}, then imply that $a_\eps(y_{h,\ell}, w_{h,\ell}) = \mathbf{W}^* A_{\eps} \mathbf{Y} =  \mathbf{w}^* A_{\eps, \ell} \mathbf{y}$. So if $\by:= A_{\eps,\ell}^{-1}R_\ell A_\eps \bV $ for some $\bV \in \mathbb{C}^n$, we have 
$$
a_\eps(y_{h,\ell}, w_{h,\ell}) \ = \ \mathbf{w}^* R_\ell A_\eps \bV \ = \ 
(R_\ell^T \mathbf{w})^*  A_\eps \bV =\ \bW^* A_\eps \bV
\ =\ a_\eps(v_h, w_{h,\ell}) , 
$$ 
where $y_{h,\ell}$ is the finite element function with nodal values $\mathbf{y}$. 
Thus, by definition of $Q_{\eps,\ell}$,  we have   $ y_{h,\ell} = Q_{\eps,\ell} v_h $,  which implies the result (i).  
 The proof of (ii) is similar. 
\end{proof}

\vspace{0.5cm}

The main results of the previous section - Theorems  \ref{thm:boundQ} 
and \ref{thm:final} -  give estimates for the norm and the field of values of the operator $Q_\eps$ on the space $\cV^h$, with respect to the inner product $(\cdot, \cdot)_{1,k}$ and its associated  norm. In the following we translate these results into norm and 
field of values estimates for the preconditioned matrix $B_{\eps, AS}^{-1}A_\eps$  in the weighted 
inner product $ \langle \cdot, \cdot\rangle_{D_k}$, where the weight matrix is   :   
\begin{equation}
 D_k := S + k^2 M\ , 
\label{eq:weight}
 \end{equation}
and  $S$ and $M$ are  defined in \eqref{eq:matrices}.  
In fact $D_k$ is the matrix representing  the $(\cdot, \cdot)_{1,k}$ 
inner product on the finite element space $\cV^h$ in the sense that if $v_h, w_h \in \cV^h$ with coefficient vectors $\bV, \bW$ then 
\begin{equation}\label{eq:ip}
(v_h, w_h )_{1,k}\  =\ \langle \bV, \bW\rangle_{D_k} . \end{equation}    
 
\begin{theorem}\label{ipnormQ}
Let $v_h = \sum_{j\in \cI^h} V_j \phi_h \ \in \ \cV^h$. Then
\begin{eqnarray*}
(i) &  (v_h, Q_\eps v_h)_{1,k} \ & =\ \langle \bV, B_{\eps, AS} ^{-1} A_\eps \bV\rangle_{D_k}\\
(ii) &  \Vert Q_\eps v_h\Vert _{1,k} \ & =\ \Vert B_{\eps, AS}^{-1} A_\eps \bV\Vert_{D_k}\\
\end{eqnarray*}
\end{theorem}
\begin{proof}
For arbitrary $w_h, v_h \in \cV^h$, with coefficient vectors $\bW$ and $\bV$, using Theorem \ref{thm:repQ}, 
we have
\beqs
(w_h, Q_{\eps,\ell}v_h)_{1,k} \ =\ \langle \bW,R_\ell^TA_{\eps,\ell}^{-1}R_\ell A_\eps \bV  \rangle_{D_k},  \ \quad \ell = 0, \ldots, N . 
\eeqs
Summing these over $\ell=0,\ldots,N$ and using \eqref{eq:minor}, \eqref{eq:coarsegrid}, and \eqref{eq:defAS}, we obtain
\beqs
(w_h, Q_\eps v_h)_{1,k} \ =\ \langle \bW,B_{\eps, AS}^{-1} A_\eps \bV  \rangle_{D_k},   
\eeqs
from which (i) and (ii) follow immediately.   
\end{proof} 

\vspace{0.5cm}

The following main result now  follows from  Theorems \ref{thm:boundQ}, \ref{thm:final}, and \ref{ipnormQ}.
\begin{theorem}[Main result for left preconditioning]\label{thm:final1}
\begin{align*}
(i)\quad \quad \Vert B_{\eps,AS}^{-1} A_\eps \Vert_{D_k}  & \ \lesssim  \ \left(\ksqeps\right) \quad \text{for all} \quad H, \Hsub.
\end{align*}
Furthermore, there exists a constant $\cC_1$ such that 
 \begin{align*}
 (ii)   \quad \quad 
 \vert \langle \bV, 
B_{\eps, AS}^{-1}A_\eps \bV \rangle_{D_k} \vert  & \ \gtrsim   \ 
\left(1 + \frac{H}{\delta}\right)^{-1} 
\left(\frac{\vert \eps\vert }{k^2} \right)^{2}\ 
\Vert \bV \Vert_{D_k}^2, \quad\tfa \bV \in \Com^n,
\end{align*}
when
\beq\label{eq:E20rpt}
\max\left\{ k\Hsub,\  kH \left(1+ \frac{H}{\delta}\right)  \kseb^2 \right\}  \ \leq \  
\cC_1 \left(1 + \frac{H}{\delta}\right)^{-1}  \left(\frac{\vert \eps \vert }{k^2}\right).
\eeq
\end{theorem}
Combining Theorem \ref{Thm:A1} and Theorem \ref{thm:final1} we obtain:
\begin{corollary}[GMRES convergence for left preconditioning]\label{cor:final2}
Consider the weighted GMRES method where the residual is minimised in
the norm induced by $D_k$ (see, e.g., \cite{GuPe:14}). 
Let $\br^m$ denote the $m$th iterate of GMRES applied to the system $A_\eps$, left  preconditioned with 
$B_{\eps,AS}^{-1}$. Then 
\begin{equation}\label{eq:conv_est}
\frac{\Vert \bfr^m \Vert_{D_k}} { \Vert \bfr^0 \Vert_{D_k} } \
\lesssim   \ \left(1- \left(1 + \frac{H}{\delta}\right)^{-2}\, \left(\frac{\vert \eps\vert }{k^2}\right)^6\right)^{m/2} \ ,  
\end{equation}
provided condition \eqref{eq:E20rpt} holds.
\end{corollary}
As a particular example of Corollary \ref{cor:final2} we see that,
provided $\vert \eps\vert  \sim k^2$, $H, \Hsub \sim k^{-1}$ and $\delta \sim H$,  then GMRES will
converge with the number of iterations independent of all parameters. 
This property is illustrated in the numerical experiments in the next
section, all of which concern the case $\delta\sim \Hsub\sim H$. These experiments also explore the sharpnesss of the result \eqref{eq:conv_est} in the two cases: 
(i) $\vert \eps\vert $ decreases below  $k^2$ for fixed $H$ and 
(ii) $H$ increases  above $k^{-1}$ for fixed $\eps$. Our experiments show that there may be  room 
to  improve  the  theoretical results. 

While Corollary \ref{cor:final2} provides rigorous  estimates only for weighted GMRES, we see in  \S\ref{sec:Numerical} that
 there is, in fact, very little  difference between the results with weighted
GMRES and standard GMRES, so most of our experiments are for standard
GMRES. The difficulty of proving results about standard GMRES in the
context of domain decomposition for non self-adjoint problems was
previously investigated by other researchers; see, e.g., \cite{CaZo:02}.    

In the next section we also explore the use of $B_{\eps, AS}^{-1}$ as a 
preconditioner for $A$. A particularly effective preconditioner is
obtained with $H \sim k^{-1} $ and $\eps \approx k$. However this
preconditioner has a complexity dominated by the cost of inverting the
coarse mesh problem. A multilevel variant where the coarse problem is
approximated by an inner GMRES iteration (within the FGMRES format) is
also proposed and is demonstrated to be very efficient for solving finite
element approximations of the Helmholtz equation with $h \sim
k^{-3/2}$. 

Some of our experiments below use right preconditioning rather than
left preconditioning. Nevertheless, using the coerciveness for the adjoint
form in Corollary \ref{cor:adjoint}, we can obtain the following  result about right
preconditioning, however in the inner product induced by
$D_k^{-1}$. From this, the analogue of Corollary
\ref{cor:final2}, with $D_k$ replaced by $D^{-1}_k$, follows.

\begin{theorem}[Main result for right preconditioning]\label{thm:final3}
With the same notation as in Theorem \ref{thm:final}, we have  
\begin{align*}
(i)\quad \quad \Vert  A_\eps B_{\eps, AS}^{-1}  \Vert_{D_k^{-1}}  & \ \lesssim  \ \left(\ksqeps\right) \quad \text{for all} \quad H, \Hsub.
\end{align*}
Furthermore, provided condition \eqref{eq:E20rpt} holds, 
 \begin{align*}
(ii) \quad\quad\vert \langle \bV, 
A_\eps B_{\eps, AS}^{-1} \bV \rangle_{D_k^{-1}} \vert  & \ \gtrsim   \ 
\left(1 + \frac{H}{\delta}\right)^{-1} 
\left(\frac{\red{|}\eps\red{|}}{k^2} \right)^{2}\ 
\Vert \bV \Vert_{D_k^{\red{-1}}}^2, \quad\tfa \bV \in \Com^n.
\end{align*}
\end{theorem}

\begin{proof}
To simplify the notation, we write $B_\eps^{-1}$ instead of $B_{\eps, AS}^{-1}$. 
An easy calculation shows that for all $\bV \in \bC^n$ and with $\bW =
D_k^{-1} \bV$, we have
$$
\frac{\vert \langle \bV, A_\eps B_\eps^{-1} \bV \rangle_{D_k^{-1}}\vert} 
{  \langle \bV , \bV\rangle_{D_k^{-1}}}\  =\ 
\frac{\vert \langle  (B^*_\eps)^{-1} A_\eps^* \bW, \bW \rangle_{D_k}\vert} 
{  \langle \bW , \bW\rangle_{D_k}}  \  =\ 
\frac{\vert \langle \bW,  (B^*_\eps)^{-1} A_\eps^* \bW\rangle_{D_k}\vert} 
{  \langle \bW , \bW\rangle_{D_k}} , $$
where    $A^*$, $(B_\eps^*)^{-1}$ are the Hermitian transposes of $A,
B_\eps^{-1}$ respectively.  
The coercivity of the adjoint form proved in Corollary
\ref{cor:adjoint} then ensures that the estimate in Theorem 
\ref{thm:final1} (ii) also holds for the adjoint matrix and the result
(ii) then follows.  The result (i) is obtained analogously from taking
the adjoint and using Theorem  \ref{thm:final1} (i).
\end{proof} 

\begin{corollary}[GMRES convergence for right preconditioning]\label{cor:final4}
Under the same assumptions,  the result of Corollary \ref{cor:final2} still 
holds when left preconditioning is replaced by right preconditioning.  
\end{corollary}

\bre[The truncated sound-soft scattering problem]\label{rem:trunc}
We now outline how the results can be adapted to hold for the truncated sound-soft scattering problem. By this, we mean the exterior, homogeneous Dirichlet problem, with the radiation condition imposed as an impedance boundary condition on a far-field boundary. That is,
\begin{subequations}\label{eq:trunc}
\begin{align}
-\Delta u - (k^2 + \ri \eps)u &= f \quad \tin \Omega,\\
\dudn - \ri \eta u &= g \quad \ton \partial \Omega_R,\\
u&=0 \quad \ton \partial \Omega_D,
\end{align}
\end{subequations}
where $\Omega_D$ is the scatterer and $\Omega_R$ is a bounded Lipschitz domain with $\Omega_D\subset \Omega_R$. With $f=0$ and an appropriate choice of $g$, the solution of the above problem is a well-known approximation to the sound-soft scattering problem (see, e.g., \cite[Problem 2.4]{GaGrSp:13} for more details).

The variational formulation of this problem is almost identical to that of the interior impedance problem in \S\ref{sec:Var}, except now the Hilbert space is $\{ v\in H^1(\Omega) : v=0 \ton \partial \Omega_D\}$ and the integrals over $\Gamma$ in \eqref{eq:Helmholtzvf_intro} and \eqref{eq:F_intro} are over $\partial \Omega_R$. 
The essential Dirichlet boundary condition means that the nodes on $\partial \Omega_D$ are no longer freedoms.
Thus, the domain decomposition technique for this problem are almost the same as those described in detail above, except that the subdomains will have Dirichlet conditions not only at interior boundaries, but also on any part of their boundary that intersects with $\partial \Omega_D$.
The rest of the results in \S\ref{sec:Var}-\S\ref{sec:Convergence} go through as before, and therefore analogues of Theorems \ref{thm:final1} and \ref{thm:final3} and Corollaries \ref{cor:final2} and \ref{cor:final4} hold for the truncated problem.
\ere

\section{Numerical Experiments}
\label{sec:Numerical}

Our numerical experiments discuss  the solution of \eqref{eq:discrete} on
the unit square, with $\eta = k$, 
discretised by the continuous linear finite element
method on a uniform triangular mesh of diameter $h$. 
The problem to be solved is thus specified by the choice of $h$ and $\eps$ which we denote by  
\begin{equation}\label{eq:prblemspec}
\hprob  \quad \text{and} \quad \epsprob\ . 
\end{equation}
\red{We will discuss the case $\epsprob>0$ in Experiment 1 (with results in Tables \ref{tab:1afinal}, \ref{tab:1bfinal}, \ref{tab:1cfinal}); the empirical observations from these results are then used to motivate a preconditioner for the pure Helmholtz problem ($\epsprob=0$) in Experiments 2 and 3 (with results in Tables \ref{tab:2afinal}, \ref{tab:2bfinal}, \ref{tab:3a}, \ref{tab:IOImpHRAS}).}

We will be focused on solving systems with     $\hprob \sim   k^{-3/2}$ 
(the discretisation level  generally believed to remove the pollution effect; see, e.g., the literature reviews in \cite[Remark 4.2]{GaGrSp:13} and \cite[\S1.2.2]{GrLoMeSp:15}),  
however the case  $\hprob \sim k$ (a fixed number of grid points per wavelength) appears as a relevant subproblem when we construct multilevel methods in Experiment 3 below.
 
In the general theory given in \S \ref{sec:DD},  
coarse grid size $H$ and subdomain  size $\Hsub$  
are  unrelated. but in our  
experiments here we construct  
local subdomains  
by taking each of the elements of the coarse grid and extending them to 
obtain an
overlapping cover with overlap parameter $\delta$. This is chosen  
as large as possible, but with the restriction no two extended subdomains can touch unless they came from touching elements of the original coarse grid.   In this scenario    
$\delta \sim H$ (generous overlap), $\Hsub \sim H$ 
and our preconditioners are thus determined by choices 
of  $H$ and $\eps$, which we denote by 
\begin{equation}\label{eq:precspec}
\Hprec  \quad \text{and} \quad \epsprec\ . 
\end{equation}

In our preconditioners the coarse grid problem is of size $\sim \Hprec^{-2}$ and there are $\sim \Hprec^{-2}$  
local problems  of size $(\Hprec/\hprob)^2$.   If there were no overlap,  the method would  be 
``perfectly load balanced''
(i.e.~local problems of the same size as the coarse problem) 
when  $\Hprec= \hprob^{1/2}$. Thus, for load balancing, 
\beq
\Hprec \sim k^{-3/4} \quad \text{when} \quad \hprob \sim k^{-3/2}. \label{eq:loadbal1}
\eeq
(Because of overlap, the local problems are larger than estimated in \eqref{eq:loadbal1}, and  the method is in fact loadbalanced for 
somewhat  finer coarse meshes than those predicted in \eqref{eq:loadbal1}.
We will investigate both cases when $\epsprob$ and $\epsprec$ 
are equal and cases when $\epsprec>\epsprob=0$. The question of greatest practical interest  
is: if $\epsprob = 0$, how to choose $\epsprec$ and $\Hprec$ 
in order to maximise the efficiency of  
the preconditioner?  This question is addressed towards the end of these experiments. but first we illustrate    the theoretical results in \S\ref{sec:Matrices} which are about the case    $\epsprec = \epsprob \not=0$.

The first preconditioner considered is the {\em Classical Additive Schwarz} ({\tt AS})
preconditioner defined in \eqref{eq:defAS}. We will also be interested in variants of this that replace the
local component  \eqref{eq:minor} with something else. 
The first variant involves averaging   
 in the overlap of the subdomains. For each fine grid node $x_j  \ \ (j \in \cI^h)$,  let
 $L_j$ denote the number of subdomains which contain  $x_j$. 
Then the local operator is:  
$$
(\matrixB_{\eps,AVE,local}^{-1}\bfv)_j  \ = \  \frac{1}{L_j} \sum _{\ell: \, x_j \in {\Omega_\ell}}
\left(R_\ell^T A_{\eps,\ell}^{-1} R_\ell\bfv\right)_j\ , \quad
\text{for each} \quad j \in \cI^h\ , 
$$
and  the corresponding 
  {\em Averaged Additive Schwarz} ({\tt AVE}) preconditioner   is: 
\begin{equation}\label{eq:defAVE}
\matrixBepsAVE  \ = \ R_0^T A_{\eps,0}^{-1}
R_0  \  +\  \matrixB_{\eps,AVE,local}^{-1} . 
\end{equation}

The second variant is  the {\em Restrictive Additive Schwarz} ({\tt RAS})
 preconditioner, which is well-known in the literature  \cite{CaSa:99}, \cite{KiSa:13}. Here  
to define the local operator,  
 for each $j \in \cI^h$, choose a single $\ell =
\ell(j)$ with the property that $x_j \in \Omega^{\ell(j)}$.  
Then  
the action of the local contribution,  for each vector of fine grid freedoms  $\bfv$, is: 
  \begin{equation}
\label{eq:RASlocal}
(\matrixB_{\eps,RAS,local}^{-1}\bfv)_j  \ = \ 
\left(R_{\ell(j)}^T A_{\eps,\ell(j)}^{-1} R_{\ell(j)}\bfv\right)_j\ , \quad
\text{for each} \quad j \in \cI^h\ ,
\end{equation}
and the RAS preconditioner is  
\begin{equation}\label{eq:defRAS}
\matrixBepsRAS  \ = \ R_0^T A_{\eps,0}^{-1}
R_0  \  +\  \matrixB_{\eps,RAS,local}^{-1}\ .
\end{equation}

All three of these variants of Additive Schwarz can be used in a {hybrid} way. 
This means that instead of doing all the local and coarse grid problems independently 
(and thus potentially in 
parallel), we  first do  a 
coarse solve and then perform the   local solves on the residual of the coarse solve. This was first introduced in  \cite{MaBr:96}. As described in \cite{GrSc:07},   
this is
closely related to the deflation method \cite{NaVu:04},  which has been used recently
to good effect in the context of shifted Laplacian combined with multigrid 
\cite{ShLaVu:13}. 
We will show results for the Hybrid RAS ({\tt HRAS}) preconditioner which takes the form 

\begin{equation}\label{eq:HRAS}
\matrixB_{\eps,{HRAS}}^{-1} :=  \matrixR_0^T \matrixA_{\eps,0}^{-1} \matrixR_0  +  \matrixP_0^T \left(  \matrixB_{\eps,RAS,local}^{-1} \right)  \matrixP_0 
\  , \end{equation}
where 
 $$\matrixP_0 \ = \ \matrixI - \matrixA_\eps \matrixR_0^T
 \matrixA_{\red{0}}^{-1} \matrixR_0\ . $$ 

All our results are obtained with GMRES without restarts, with the implementation done in python.
The results
in Experiment 1  
are obtained both with  left preconditioning  and with  
right preconditioning (flexible GMRES); the relevant python codes are   
{\tt scipy.sparse.linalg.gmres} and {\tt pyamg.krylov.fgmres} respectively.
The other two experiments are done with right preconditioning.
In all experiments we use standard GMRES, which minimises the residual
in the standard Euclidean inner product. However, motivated by Theorem  
\ref{thm:final1} and Corollary \ref{cor:final2},  Experiment 1 
also discusses the results of applying  left preconditioned GMRES  in the inner product induced by the matrix 
$D_k$. For this we use the algorithm described in \cite{GuPe:14}, editing an existing GMRES code to work with the inner product induced by $D_k$.  
In all experiments the
starting guess is zero and the   
residual reduction tolerance is set at 
$10^{-6}$.   

\subsection{Experiment 1\  }

\begin{table}[h]
\begin{center}

\begin{tabular}{|c||c|c|c|c|}
\hline 
& \multicolumn{4}{|c|}{Left preconditioning}\tabularnewline
\hline
& \multicolumn{4}{|c|}{$\alpha=1$}\tabularnewline
\hline 
$k$ & $\#_{AS}$ & $\#_{AVE}$ & $\#_{RAS}$ & $\#_{HRAS}$\tabularnewline
\hline 
10 &  21 & 15 & 15 &         8\tabularnewline
20 & 20 & 15 & 15 &         8\tabularnewline
40 & 21 & 16 & 16 &         9\tabularnewline
60 & 21 & 16 & 16 &         9\tabularnewline
80 & 26 & 18 & 16 &         9\tabularnewline
100 & 21 & 17 & 16 &       9\tabularnewline
\hline 
\multicolumn{5}{|c|}{$\alpha=0.9$}\tabularnewline
\hline 
$k$ & $\#_{AS}$ & $\#_{AVE}$ & $\#_{RAS}$ & $\#_{HRAS}$\tabularnewline
\hline 
10 & 19 & 15 & 15 &         8\tabularnewline
20 &  23 & 18 & 18 &         9\tabularnewline
40 &  27 & 21 & 19 &        10\tabularnewline
60 &  25 & 20 & 20 &        10\tabularnewline
80 & 25 & 21 & 20 &        10\tabularnewline
100 &  25 & 21 & 20 &     10\tabularnewline
\hline 
\multicolumn{5}{|c|}{$\alpha=0.8$}\tabularnewline
\hline 
$k$ &  $\#_{AS}$ & $\#_{AVE}$ & $\#_{RAS}$ & $\#_{HRAS}$\tabularnewline
\hline 
10 &  19 & 15 & 14 &         8\tabularnewline
20 &  21 & 18 & 17 &         9\tabularnewline
40 & 23 & 22 & 19 &        10\tabularnewline
60 &  21 & 20 & 19 &        11\tabularnewline
80 & 21 & 20 & 19 &        11\tabularnewline
100 &  22 & 23 & 19 &     11\tabularnewline
\hline 
\end{tabular}~%
\begin{tabular}{|c|c|c|c|}
\hline 
\multicolumn{4}{|c|}{Right preconditioning}\tabularnewline
\hline
\multicolumn{4}{|c|}{ $\alpha=1$}\tabularnewline
\hline 
$\#_{AS}$ & $\#_{AVE}$ & $\#_{RAS}$ & $\#_{HRAS}$\tabularnewline
\hline 
  21 &  15 &  15 &         8\tabularnewline
  19 &  15 &  15 &         8\tabularnewline
  19 &  16 &  15 &         8\tabularnewline
  19 &  16 &  15 &         8\tabularnewline
  23 &  17 &  15 &         8\tabularnewline
  19 &  16 &  15 &         8\tabularnewline
\hline 
\multicolumn{4}{|c|}{ $\alpha=0.9$}\tabularnewline
\hline 
$\#_{AS}$ & $\#_{AVE}$ & $\#_{RAS}$ & $\#_{HRAS}$\tabularnewline
\hline 
  19 &  15 &  15 &         8\tabularnewline
  21 &  18 &  17 &         8\tabularnewline
  24 &  19 &  17 &         9\tabularnewline
  21 &  20 &  18 &         9\tabularnewline
  21 &  20 &  18 &         9\tabularnewline
  21 &  20 &  18 &         9\tabularnewline
\hline 
\multicolumn{4}{|c|}{ $\alpha=0.8$}\tabularnewline
\hline 
 $\#_{AS}$ & $\#_{AVE}$ & $\#_{RAS}$ & $\#_{HRAS}$\tabularnewline
\hline 
  18 &  15 &  14 &         8\tabularnewline
  20 &  18 &  17 &         9\tabularnewline
  20 &  20 &  17 &        10\tabularnewline
  18 &  19 &  17 &        10\tabularnewline
  18 &  19 &  17 &        10\tabularnewline
  18 &  20 &  17 &        10\tabularnewline
\hline 
\end{tabular}
\end{center}
\caption{ Number of iterations for various preconditioners with  $\hprob =
  k^{-3/2}$,    $\epsprob = \epsprec = k^2$,   $\Hprec = k^{-\alpha}$ }
\label{tab:1afinal} 
\end{table}

Here we  solve  \eqref{eq:discrete} with $\bf{f}= \bf{1}$, $\hprob = k^{-3/2}$ ,with various choices of  $ \epsprob = \epsprec$ and 
 $\Hprec$. We     use
triangular coarse grids. The results  in Section
\ref{sec:Matrices}
tell us that,  
provided $$\epsprob = \epsprob\sim k^2 \quad \text{and} \quad k \Hprec \sim 1, $$ 
then the number of GMRES iterations 
with the  preconditioner {\tt AS}  
will remain bounded as $k \rightarrow \infty$. 
Our first set of results are for the regular (Euclidean inner product) GMRES algorithm.  
 Tables \ref{tab:1afinal}, \ref{tab:1bfinal} and  \ref{tab:1cfinal}  
give  results for  a range of $\epsprob = \epsprec$ and $\Hprec$, assuming that $\Hprec=k^{-\alpha}$ for different choices of $\alpha$. 
The number of iterations of the Classical Additive Schwarz method is denoted by    $\#_{AS}$, while the number of iterations  
for the variants using  averaging, RAS and Hybrid RAS  
 are denoted 
$\#_{AVE}$,   $\#_{RAS}$ and $\#_{HRAS}$.

In Table \ref{tab:1afinal},  the results  for $\alpha = 1$  confirm  
the  result of Corollaries \ref{cor:final2} and \ref{cor:final4}. The other parts of this  
table show that  in fact when $\epsprob = \epsprec =  k^2$ then 
the iteration counts remain bounded as $k$
increases for a range of $\Hprec$ chosen to decrease more slowly with $k$ 
than  the theoretical requirement of   $\mathcal{O}(k^{-1})$. Thus,  if
there is enough absorption, the preconditioner still works well for 
solving the shifted system,   even for much  coarser coarse meshes than those predicted by Corollaries \ref{cor:final2} and \ref{cor:final4}. 
The case $\Hprec = k^{-0.8}$ is close to 
being load balanced (see \eqref{eq:loadbal1} and the remarks following). To give an idea of the sizes of the systems involved, when $\Hprec = k^{-0.8} $ and $k = 120$ the size of the fine grid problem 
is $n=1782225$ while the size of the coarse grid 
problem is $2116$ and there are  $2025$ local
problems of maximal size $3364$ to be solved. 
We also note the overall improvement as we compare different  preconditioners in the sequence 
{\tt AS},  {\tt AVE}, {\tt RAS}, {\tt HRAS}.  

Then Tables \ref{tab:1bfinal} and \ref{tab:1cfinal} repeat the same
experiments for the cases $\epsprob = \epsprec = k $ and $\epsprob = \epsprec = 1 $. We observe
here that when $\Hprec = k^{-1}$,   both 
methods continue to  work quite well (although the number of
iterations does grow  mildly with $k$), however for coarser coarse meshes  
$\Hprec = k^{-\alpha}$
 with $\alpha < -1$,  the method quickly becomes unusable.   The general superiority of {\tt
  HRAS} over the other methods is striking. A * in the tables indicates that the number of iterations was above
200.

\begin{table}[h]
\begin{center}

\begin{tabular}{|c||c|c|c|c|}
\hline 
& \multicolumn{4}{|c|}{Left preconditioning}\tabularnewline
\hline
& \multicolumn{4}{|c|}{$\alpha=1$}\tabularnewline
\hline 
$k$ & $\#_{\tt AS}$ & $\#_{\tt AVE}$ & $\#_{\tt RAS}$ & $\#_{\tt HRAS}$\tabularnewline
\hline 
10 & 25  & 16 & 16 &         10\tabularnewline
20 & 25 & 20 & 20 &         11\tabularnewline
40 & 35 & 30 & 31 &         16\tabularnewline
60 & 46 & 41 & 42 &         22\tabularnewline
80 & 67 & 56 & 55 &         30\tabularnewline
100 & 75& 69 & 70 &        38\tabularnewline
\hline 
\multicolumn{5}{|c|}{$\alpha=0.9$}\tabularnewline
\hline 
$k$ & $\#_{\tt AS}$ & $\#_{\tt AVE}$ & $\#_{\tt RAS}$ & $\#_{\tt HRAS}$\tabularnewline
\hline 
10 & 22 & 18 & 18 & 10         \tabularnewline
20 &  37 & 27 & 28 & 14         \tabularnewline
40 &  118 & 45 & 85 &        24\tabularnewline
60 &  * & 171 & 192 &        40\tabularnewline
80 & * & * & * &        61\tabularnewline
100 &*  & *& *&        97\tabularnewline
\hline 
\multicolumn{5}{|c|}{$\alpha=0.8$}\tabularnewline
\hline 
$k$ &  $\#_{\tt AS}$ & $\#_{\tt AVE}$ & $\#_{\tt RAS}$ & $\#_{\tt HRAS}$\tabularnewline
\hline 
10 &  23 & 18 & 19 &         12\tabularnewline
20 &  40 & 33 & 35 &         18\tabularnewline
40 & 173 & 140 & 190 &        122\tabularnewline
60 &  * & * & * &        *\tabularnewline
80 & * & * & * &        *\tabularnewline
100 & *  & * & * & *\tabularnewline
\hline 
\end{tabular}~%
\begin{tabular}{|c|c|c|c|}
\hline 
\multicolumn{4}{|c|}{Right preconditioning}\tabularnewline
\hline
\multicolumn{4}{|c|}{ $\alpha=1$}\tabularnewline
\hline 
$\#_{\tt AS}$ & $\#_{\tt AVE}$ & $\#_{\tt RAS}$ & $\#_{\tt HRAS}$\tabularnewline
\hline 
  25 &  17 &  16 &         9\tabularnewline
  24 &  20 &  20 &         11\tabularnewline
  32 &  28 &  28 &         14\tabularnewline
  40 &  38 &  37 &         19\tabularnewline
  56 &  50 &  48 &         24\tabularnewline
  64 &  63 &  61 &          31\tabularnewline
\hline 
\multicolumn{4}{|c|}{ $\alpha=0.9$}\tabularnewline
\hline 
$\#_{\tt AS}$ & $\#_{\tt AVE}$ & $\#_{\tt RAS}$ & $\#_{\tt HRAS}$\tabularnewline
\hline 
  22 &  18 &  18&         10\tabularnewline
  35 &  27 &  27 &         13\tabularnewline
  107 &  42 &  83 &         21\tabularnewline
  \hspace{0ex}* &  175 &  187 &         35\tabularnewline
  \hspace{0ex}* &  * &  * &         54\tabularnewline
  \hspace{0ex}*& * & * & 86\tabularnewline
\hline 
\multicolumn{4}{|c|}{ $\alpha=0.8$}\tabularnewline
\hline 
 $\#_{\tt AS}$ & $\#_{\tt AVE}$ & $\#_{\tt RAS}$ & $\#_{\tt HRAS}$\tabularnewline
\hline 
  23 &  18 &  18 &         12\tabularnewline
  37 &  33 &  34 &         17\tabularnewline
  153 &  130 &  177 &        116\tabularnewline
    \hspace{0ex}* &  * &  * &        *\tabularnewline
    \hspace{0ex}* &  * &  * &        *\tabularnewline
  \hspace{0ex}*  & * &*  &* \tabularnewline
\hline 
\end{tabular}
\end{center} 
\caption{Number of iterations for various preconditioners with $\hprob =
  k^{-3/2}$,   $\epsprob = \epsprec = k$,    $\Hprec = k^{-\alpha}$ }
\label{tab:1bfinal}
\end{table}

\begin{table}[h]
\begin{center}
\begin{tabular}{lr}
\begin{tabular}{|c||c|c|c|c|}
\hline 
& \multicolumn{4}{|c|}{Left preconditioning}\tabularnewline
\hline
& \multicolumn{4}{|c|}{$\alpha=1$}\tabularnewline
\hline 
$k$ & $\#_{\tt AS}$ & $\#_{\tt AVE}$ & $\#_{\tt RAS}$ & $\#_{\tt HRAS}$\tabularnewline
\hline 
10 &  26 & 17 & 17 &         10\tabularnewline
20 & 25 & 21 & 21 &         12\tabularnewline
40 & 37 & 32 & 32 &         17\tabularnewline
60 & 49 & 44 & 45 &         24\tabularnewline
80 & 74 & 63 & 61 &         33\tabularnewline
100 &  83 & 78 & 79 &      43 \tabularnewline
\hline 
\multicolumn{5}{|c|}{$\alpha=0.9$}\tabularnewline
\hline 
$k$ & $\#_{\tt AS}$ & $\#_{\tt AVE}$ & $\#_{\tt RAS}$ & $\#_{\tt HRAS}$\tabularnewline
\hline 
10 & 23 & 18 & 18 &         10\tabularnewline
20 &  38 & 28 & 30 &         14\tabularnewline
40 &  134  & 49  & 96 &        26\tabularnewline
60 &  *  & * & * &        44\tabularnewline
80 & * & * & * &        71\tabularnewline
100 & *  & * & * & 115 \tabularnewline
\hline 
\multicolumn{5}{|c|}{$\alpha=0.8$}\tabularnewline
\hline 
$k$ &  $\#_{\tt AS}$ & $\#_{\tt AVE}$ & $\#_{\tt RAS}$ & $\#_{\tt HRAS}$\tabularnewline
\hline 
10 &  23 & 20 & 19 &         13\tabularnewline
20 &  42 & 35 & 39 &         19\tabularnewline
40 & *  & 194 & * &        182\tabularnewline
60 &  * & * & * &        *\tabularnewline
80 & * & * & * &        *\tabularnewline
100 & *  &* & * & *\tabularnewline
\hline 
\end{tabular}~%
&
\begin{tabular}{|c|c|c|c|}
\hline 
\multicolumn{4}{|c|}{Right preconditioning}\tabularnewline
\hline
\multicolumn{4}{|c|}{ $\alpha=1$}\tabularnewline
\hline 
$\#_{\tt AS}$ & $\#_{\tt AVE}$ & $\#_{\tt RAS}$ & $\#_{\tt HRAS}$\tabularnewline
\hline 
  26 &  17 &  17 &         10\tabularnewline
  24 &  21 &  20 &         11\tabularnewline
  33 &  30 &  30 &         15\tabularnewline
  43 &  41 &  40 &         20\tabularnewline
  62  &  56 &  53 &         27\tabularnewline
  71&  70&  68& 34 \tabularnewline
\hline 
\multicolumn{4}{|c|}{ $\alpha=0.9$}\tabularnewline
\hline 
$\#_{\tt AS}$ & $\#_{\tt AVE}$ & $\#_{\tt RAS}$ & $\#_{\tt HRAS}$\tabularnewline
\hline 
  22 &  18 &  18 &         10\tabularnewline
  37 &  28 &  28 &         13\tabularnewline
  121 &  45  &  93 &         23\tabularnewline
    \hspace{0ex}* &  * &  * &         39\tabularnewline
  \hspace{0ex}* &  * &  * &         62\tabularnewline
  \hspace{0ex}*  & * & * & 101\tabularnewline
\hline 
\multicolumn{4}{|c|}{ $\alpha=0.8$}\tabularnewline
\hline 
 $\#_{\tt AS}$ & $\#_{\tt AVE}$ & $\#_{\tt RAS}$ & $\#_{\tt HRAS}$\tabularnewline
\hline 
  23 &  20 &  19 &         13\tabularnewline
  40 &  35 &  36 &         17\tabularnewline
  187  &  186  &  * &        181\tabularnewline
  \hspace{0ex}* &  * &  * &        *\tabularnewline
  \hspace{0ex}*  &  * &  * &        *\tabularnewline
  \hspace{0ex}*  & * & * & * \tabularnewline
\hline 
\end{tabular}
\end{tabular}
\end{center} 
\caption{Number of iterations for various preconditioners with
$\hprob =  k^{-3/2}$,   $\epsprob = \epsprec = 1$,   $\Hprec = k^{-\alpha}$}
\label{tab:1cfinal}
\end{table}

To finish Experiment 1, we  repeated the experiments above with left
preconditioning but where the GMRES algorithm minimises the residual
in the norm induced by $D_k$. The resulting iteration counts  
were almost identical to  those  given in Tables \ref{tab:1afinal}, \ref{tab:1bfinal} and \ref{tab:1cfinal}, so we do not give them here. 

\red{Note that the results in Table \ref{tab:1bfinal}, especially \red{the columns corresponding to}  
$\alpha=1$ and HRAS, show that $B^{-1}_k$ is a good (although admittedly not perfect) preconditioner for $A_k$. Based on \cite{GaGrSp:13} this strongly suggests $B^{-1}_k$ will be a good preconditioner for $A$ (recall properties (i) and (ii) in the introduction); we see that this is indeed the case in the next experiment which is about preconditioners for $A$.}

While Experiment 1 illustrates very well  
our theoretical results about preconditioning the problem  with absorption 
(i.e.~the matrix $A_\eps$), the ultimate goal
of  our  work 
is to  determine the best preconditioner for the  problem 
without absorption (i.e.~the matrix $A$).  
Therefore, the rest of our experiments focus on investigating  this question, and so from now on
we take $\eps_{\mathrm{prob}} = 0$. 
Also, in the  experiments above,   {\tt HRAS} outperformed  all the other 
preconditioners and so  
in the rest of our experiments we restrict attention to  {\tt HRAS}. 

Moreover, remembering that the local solves in $B_{\eps,RAS,local}^{-1}$ are
solutions of local problems with a Dirichlet condition on interior
boundaries of subdomains, and noting that these are not expected to perform 
well for genuine wave propagation (i.e.~$\eps$  small), 
we also  consider the use of impedance boundary  conditions on the local solves.
We therefore  
introduce the  sesquilinear form local to the subdomain $\Omega_\ell$,
defined as the following local equivalent of
\eqref{eq:Helmholtzvf_intro}:  
$$ a_{\eps,Imp,\ell}(v,w) \ = \ \int_{\Omega_\ell} \nabla v . \nabla
\overline{w}  -    (k^2 + i \eps) \int_{\Omega_\ell} v \overline{w} - \ri
k \int_{\partial \Omega_\ell} v \overline{w}, 
$$
(remember that we are choosing $\eta = k$ in all the experiments). We let $A_{\eps,Imp,\ell}$ be the stiffness matrix arising from this form,
i.e.
$$(A_{\eps,Imp,\ell})_{j,j'} \ = \ a_{\eps,Imp, \ell}(\phi_{j'}, \phi_{j})\ ,
\quad j,j' \in \cI(\overline{\Omega_\ell}) .$$ 
This can be used as a local operator in any of the preconditioners 
introduced above. For example if it is inserted into the {\tt HRAS} operator 
\eqref{eq:HRAS},  then the one-level variant  is 
  \begin{equation}
\label{eq:ImpRASlocal}
(\matrixB_{\eps,Imp,RAS,local}^{-1}\bfv)_j  \ = \ 
\left(\widetilde{R}_{\ell(j)}^T A_{\eps,Imp,\ell(j)}^{-1} \widetilde{R}_{\ell(j)}\bfv\right)_j\ , \quad
\text{for each} \quad j \in \cI^h\ .
\end{equation}
\red{Here (noting the distinction with \eqref{eq:minor}), $\widetilde{R}_\ell$ denotes the 
restriction operator $(\tilde{R}_{\ell})_{j,j'} = \delta_{j,j'}$,  (as before)   $j'$ ranges  
over all  $\cI^h$, but now   
$j$ runs over all  indices such that  $x_j \in \overline{\Omega}_\ell \backslash{\Gamma}$.} The hybrid two-level variant is 
\begin{equation}\label{eq:HRASImp}
\matrixB_{\eps,{Imp, HRAS}}^{-1} :=  \matrixR_0^T \matrixA_{\eps,0}^{-1} \matrixR_0  +  \matrixP_0^T \left(  \matrixB_{\eps,Imp,RAS,local}^{-1} \right)  \matrixP_0 
\  . \end{equation}
We refer to these as the one- and two-level {\tt ImpHRAS} preconditioners. 
\vspace{0.5cm}

\subsection{Experiment 2}   
In Tables \ref{tab:2afinal} and \ref{tab:2bfinal} below,  we illustrate the performance of these preconditioners 
with various choices of $\epsprec$  when solving  the problem \eqref{eq:discrete}
with $\epsprob = 0$ and $\hprob = k^{-3/2}$. 
Here we   use rectangular coarse grids and subproblems and employ right 
(FGMRES) preconditioning (although the performance with triangular grids and left 
preconditioning is similar). 
    In order to ensure our  problem \eqref{eq:discrete} 
has  physical significance,   
we choose the data $f,g$ in \eqref{eq:F_intro}  
so that the exact solution of problem \eqref{eq:vf_intro} is a plane wave $u(x) = \exp(\ri k x.\hat{d})$ where $\hat{d} = (1/\sqrt{2},1/\sqrt{2})^T$.  
In Tables \ref{tab:2afinal} and \ref{tab:2bfinal} respectively, 
iteration counts for the two level versions of {\tt HRAS} and {\tt ImpHRAS}  
are given, {with the counts for the corresponding one level method 
given as  subscripts}.  From these tables we make the following observations: 
\begin{enumerate}
\item When $\Hprec = k^{-1}$, the two-level versions of both   {\tt HRAS} and {\tt ImpHRAS} 
perform quite well, although the number of iterations does  grow mildly with $k$.  The corresponding one-level versions perform  poorly, showing that the coarse grid operator is doing a good job in this scenario. The choice of $\epsprec$ has minimal effect (except that it appears that 
 $\epsprec$ should not be chosen much bigger than  $ k^{1.5}$.  
\item   When $\Hprec = k^{-\alpha}$ and $\alpha<1$, 
  {\tt HRAS} becomes unusable. {\tt ImpHRAS} also degrades as $\alpha$ decreases but 
then starts to improve again, and   at $\alpha = 0.6$ provides a
reasonably efficient solver with very slow growth of iterations with
$k$. Here  the two-grid variant is not 
much better than the one-grid variant,  due to the fact that the coarse grid problem 
has become very coarse (when $k = 80$, $n= 511225$,  the size of the
coarse grid problem is $196$,  and the size of each of the largest  
local problem is $10404$. 
\end{enumerate}

\begin{table}
\begin{center}
\begin{tabular}{|c||c|c|c|c|c|c|c|}
\hline 
\multicolumn{8}{|c|}{$\alpha=1$}\tabularnewline
\hline 
$k\backslash\beta$ & 0  & 0.4  & 0.8 & 1 & 1.2  & 1.6 & 2.0\tabularnewline
\hline 
10 & $10_{34}$  & $10_{34}$   & $11_{34}$ & $11_{34}$ & $12_{34}$  & $15_{33}$  & $19_{34}$\tabularnewline
20 & $12_{92}$ & $12_{92}$ & $12_{92}$ & $12_{92}$ & $13_{92}$ &  $19_{92}$ &  $37_{93}$\tabularnewline
40 &  $18_{*}$ & $18_{*}$ & $18_{*}$ & $18_{*}$ & $18_{*}$ &  $25_{*}$ & $63_{*}$\tabularnewline
60  & $25_{*}$ & $25_{*}$ & $25_{*}$ & $25_{*}$ & $25_{*}$ &  $32_{*}$ & $86_{*}$\tabularnewline
80 & $34_{*}$ & $34_{*}$ &  $33_{*}$ & $33_{*}$ & $32_{*}$  & $39_{*}$ & $110_{*}$\tabularnewline
100 &$45_{*}$ & $45_{*}$ & $44_{*}$ & $43_{*}$ & $42_{*}$ & $47_{*}$ & $136_{*}$  \tabularnewline
\hline 
\hline 
\multicolumn{8}{|c|}{$\alpha=0.8$}\tabularnewline
\hline 
$k\backslash\beta$ & 0 &  0.4 &  0.8 & 1 & 1.2 &  1.6 &  2.0\tabularnewline
\hline 
10 & $16_{25}$ &  $16_{24}$ & $16_{24}$ &  $16_{24}$ & $17_{24}$  & $18_{24}$ &  $19_{26}$\tabularnewline
20 & $23_{59}$ &  $23_{59}$ & $23_{59}$ &  $23_{59}$ & $24_{59}$ &  $30_{58}$ &  $39_{61}$\tabularnewline
40 & $*_{*}$ & $*_{*}$ &  $175_{*}$ & $139_{*}$ & $96_{*}$ &  $65_{160}$ & $76_{132}$\tabularnewline
60 & $*_{*}$  & $*_{*}$ &  $*_{*}$ & $*_{*}$ & $169_{*}$ &  $114_{*}$ &  $119_{197}$\tabularnewline
80 & $*_{*}$ &  $*_{*}$ &  $*_{*}$ & $*_{*}$ & $*_{*}$ & $166_{*}$ &  $165_{*}$\tabularnewline
100 &$*_{*}$ &$*_{*}$ &$*_{*}$ &$*_{*}$&$*_{*}$&$*_{*}$&$*_{*}$ \tabularnewline
\hline 
\hline 
\multicolumn{8}{|c|}{$\alpha=0.6$}\tabularnewline
\hline 
$k\backslash\beta$ & 0 &  0.4 &  0.8 & 1 & 1.2 &  1.6 &  2.0\tabularnewline
\hline 
10 & $16_{18}$ &  $16_{18}$ & $16_{18}$ &  $16_{18}$ & $16_{18}$ &  $16_{19}$ & $19_{22}$\tabularnewline
20 & $55_{72}$ &  $55_{71}$ &  $53_{67}$ & $51_{63}$ & $48_{58}$ &  $39_{46}$ &  $39_{43}$\tabularnewline
40 & $140_{124}$ &  $138_{129}$ &  $131_{135}$ & $125_{133}$ & $114_{125}$ &  $86_{94}$ &  $81_{76}$\tabularnewline
60 & $*_{*}$ &  $*_{*}$ &  $*_{*}$ & $*_{*}$ & $*_{*}$ & $147_{141}$ &  $113_{102}$\tabularnewline
80 & $*_{*}$ &  $*_{*}$ &  $*_{*}$ & $*_{*}$ & $*_{*}$  & $178_{160}$ &  $135_{121}$\tabularnewline
100 & $*_{*}$& $*_{*}$ & $*_{*}$ & $*_{*}$ &&& \tabularnewline
\hline 
\end{tabular}

\end{center} 
\caption{Number of iterations for ${\tt HRAS}$ with
  $\epsprob = 0$, $\epsprec = k^{\beta}$,   $\Hprec = k^{-\alpha}$, and right preconditioning.
 } 
\label{tab:2afinal}
\end{table}

\begin{table}
\begin{center}
\begin{tabular}{|c||c|c|c|c|c|c|c|}
\hline 
\multicolumn{8}{|c|}{$\alpha=1$}\tabularnewline
\hline 
$k\backslash\beta$ & 0 &  0.4 &  0.8 & 1 & 1.2 & 1.6 &  2.0\tabularnewline
\hline 
10 & $15_{45}$ &   $15_{45}$ & $15_{45}$ & $15_{46}$ & $15_{46}$ &  $17_{47}$ &  $21_{49}$\tabularnewline
20 & $17_{104}$ & $17_{104}$ &  $17_{105}$ & $17_{105}$ & $18_{105}$ &  $20_{107}$ & $34_{113}$\tabularnewline
40 &  $21_{*}$ & $21_{*}$ & $21_{*}$ & $21_{*}$ & $21_{*}$ &  $26_{*}$ & $56_{*}$\tabularnewline
60 & $27_{*}$ & $27_{*}$ &  $27_{*}$ & $27_{*}$ & $27_{*}$ &  $33_{*}$ &  $78_{*}$\tabularnewline
80 & $36_{*}$ &  $36_{*}$ & $35_{*}$ & $35_{*}$ & $34_{*}$ & $40_{*}$ & $101_{*}$\tabularnewline
100 & $47_{*}$ & $47_{*}$ & $46_{*}$ & $45_{*}$ & $43_{*}$ & $48_{*}$ & $123_{*}$ \tabularnewline
\hline 
\hline 
\multicolumn{8}{|c|}{$\alpha=0.8$}\tabularnewline
\hline 
$k\backslash\beta$ & 0 & 0.4 &  0.8 & 1 & 1.2 & 1.6 &  2.0\tabularnewline
\hline 
10 & $15_{27}$ &  $15_{27}$ & $15_{27}$ &  $15_{27}$ & $15_{27}$ & $16_{28}$ &  $19_{29}$\tabularnewline
20 & $23_{56}$ & $23_{56}$ & $23_{56}$ &  $23_{56}$ & $23_{56}$ &  $25_{57}$ &  $34_{62}$\tabularnewline
40 & $51_{106}$ &  $51_{106}$ &  $51_{106}$ & $51_{106}$ & $49_{106}$ &  $48_{108}$ & $63_{116}$\tabularnewline
60 & $107_{150}$ &  $106_{150}$ &  $105_{150}$ & $104_{150}$ & $99_{150}$ &  $85_{151}$ &  $99_{164}$\tabularnewline
80 & $187_{194}$ & $185_{193}$ &  $183_{193}$ & $178_{193}$ & $168_{193}$ &  $132_{194}$ &  $138_{*}$\tabularnewline
100 & $*_{*}$ & $*_{*}$  & $*_{*}$  & $*_{*}$  & $*_{*}$  &$185_{*}$  & $179_{*}$ \tabularnewline
\hline 
\hline 
\multicolumn{8}{|c|}{$\alpha=0.6$}\tabularnewline
\hline 
$k\backslash\beta$ & 0 &  0.4 &  0.8 & 1 & 1.2 &  1.6 &  2.0\tabularnewline
\hline 
10 & $14_{18}$ &  $14_{18}$ &  $14_{18}$ & $14_{18}$ & $14_{19}$ &  $15_{20}$ &  $17_{23}$\tabularnewline
20 & $27_{31}$ &  $27_{31}$ &  $26_{31}$ & $26_{31}$ &  $26_{32}$ & $28_{33}$ &  $36_{42}$\tabularnewline
40 & $51_{51}$ &  $51_{51}$ &  $50_{51}$ & $50_{51}$ & $48_{51}$ &  $50_{51}$ &  $73_{66}$\tabularnewline
60 & $72_{71}$ &  $71_{71}$ &  $70_{71}$ & $69_{71}$ & $69_{70}$ &  $70_{67}$ &  $104_{91}$\tabularnewline
80 & $74_{85}$ &  $74_{85}$ & $74_{85}$ &  $74_{84}$ & $74_{83}$ &  $77_{77}$ &  $126_{111}$\tabularnewline
100 &$84_{98}$& $84_{98}$ & $84_{98}$ & $84_{97}$ & $84_{95}$ & $86_{87}$ & $148_{131}$\tabularnewline
\hline 
\end{tabular}
\end{center} 
\caption{Number of iterations for ${\tt ImpHRAS}$ with
  $\epsprob = 0$, $\epsprec = k^{\beta}$,   $\Hprec = k^{-\alpha}$, and right preconditioning.
}
\label{tab:2bfinal} 
\end{table}

\red{At the end of Experiment 1, we argued that the fact that $B^{-1}_k$ was a good preconditioner for $A_k$, along with the results of \cite{GaGrSp:13}, suggested that $B_{k}^{-1}$ would be a good preconditioner for $A$. 
Having performed Experiment 2, we can now compare preconditioning $A_k$ with $B_k^{-1}$ to preconditioning $A$ with $B_k^{-1}$.
Relevant results are in the column in Table \ref{tab:1bfinal} for right-preconditioning with HRAS and $\alpha=1$ and the column in Table \ref{tab:2afinal} with $\alpha=1$ and $\beta=1$.
Although the iteration counts are slightly different, a linear least-squares fit shows that the rate of growth with $k$ is very similar in each case ($k^{0.60}$ versus $k^{0.53}$)
which is in line with the intuition above.
}

\subsection{Experiment 3}

From observations above, we can identify a possible multilevel
strategy for preconditioning the problem with $\hprob = k^{-3/2}$ and
$\epsprob = 0$.  We do this only in 2D using the experiments above,
but it is possible to carry out a similar analysis in 3D. 

The first step is to use a 
two level {\tt HRAS} with,  say,   $\epsprec =  k$ and $\Hprec =
\mathcal{O}(k^{-1})$.  Denote the number of iterations of this method
by   $I_1(k)$. (A least squares linear fit of the data for the two-level method in Column 5 of the
first pane of   Table
\ref{tab:2afinal},  indicates that $I_1(k) \approx
\mathcal{O}(k^{0.4})$.)   
  Each application of the preconditioner requires the solution of one system of size   
$\Hprec^{-2} = \mathcal{O}(k^2)$ and $\cO(k^2)$ systems of size
$(\Hprec/\hprob)^2 \approx \cO(k)$. Each matrix-vector
multiplication with the system matrix needs  $\mathcal{O}(k^3)$
operations.    The total cost is then approximately:
\begin{equation}
\label{eq:approxcost}
I_1(k) \left[\underbrace{C(k^2)}_{\mathrm{Coarse\  solve}} + 
\underbrace{\cO(k^2)  C(k)}_{\mathrm{local \ solves}} +
  \underbrace{\cO(k^3)}_{\mathrm{matrix-vector \
      multiplications}}\right] \ ,   
\end{equation}
where $C(m)$ denotes an estimate of the cost of backsolving with a
factorized  $m \times m$ finite element system in 2D (the theoretical upper bound is    $C(m) \sim m^{3/2}$).
(Only
backsolves need be counted since these appear in every iteration while
the  factorization needs only be done
once.) 

The local solves in \eqref{eq:approxcost}  can in
principle   be done  in parallel,
so that the main bottleneck is likely to be  the
(relatively large) coarse solve. For this reason we consider
replacing  the direct  coarse solve  with  an inner iteration within an FGMRES
 set-up.
Thus we need an efficient iterative method for the coarse problem,
which itself is a  Helmholtz
finite element system with $\hprob \sim   k^{-1}$, $\epsprob = k$. Table
\ref{tab:3a} gives some experiments with preconditioned iterative
methods for  problems of this form  in the two cases 
$\hprob = \pi/10k$ and $\hprob = \pi/5k$ (20 and 10 grid points per
wavelength respectively), using {\tt ImpHRAS} with $\epsprec = k$ and $\Hprec =
k^{-1/2}$.  Iteration counts are given   
for the two-level variant (and  the one-level variant as subscripts).     
These show that the one-level method works just as well
as (sometimes even better than) the two level method. With the number of iterations denoted by
$I_2(k)$, a least-squares linear fit to the one-level data for $\hprob= \pi/5k$ yields the estimate $I_2(k) \sim k^{0.3}$. 

\begin{table}[h]
\begin{center}
\begin{tabular}{|c||c|c|}
\hline 
$k$ & $\hprob = \pi/5k$  & $\hprob = \pi/10k$ \\
\hline 
10 & $9_{10}$ & $9_{9}$ \\
20 &  $14_{15}$ & $14_{15}$ \\
40 &  $21_{24}$ & $22_{24}$ \\
60 &  $30_{32}$ & $31_{32}$ \\
80 &  $35_{35}$ & $37_{35}$\\
100 &   $39_{38}$ & $41_{39}$ \\
120 &  $42_{40}$ & $45_{43}$ \\
140 &  $46_{43}$ & $49_{46}$ \\
\hline 
\end{tabular}
\end{center}
\caption{Number of iterations for {\tt ImpHRAS} with $\epsprob = k = \epsprec$,
  $\Hprec = k^{-0.5}$, and right preconditioning.}
\label{tab:3a}
\end{table}

If we use this method (in its one level form) to approximate the
coarse solve 
in \eqref{eq:approxcost},  then each application of the preconditioner requires  
the solution of $\cO(k)$ systems of size $\cO(k)$ and the total cost
of the solution 
is about  $I_2(k)\left(kC(k) + \cO(k^2)\right)$. 
Using this to replace the
first term on the right-hand side of \eqref{eq:approxcost}, 
the cost  of the resulting
inner-outer algorithm  would be
approximately

\begin{equation}
\label{eq:approxcost1}
I_1(k) 
\left[
I_2(k) \left(k C(k) + k^2\right) + 
\cO(k^2)  C(k) +
  \cO(k^3)
\right] \ .   
\end{equation}

Now it is well-known that in 2D fast direct solvers for finite element systems of
size $k$ perform with $\cO(k)$ complexity for $k \leq 10^5$. So for
practically relevant wavenumbers we expect a complexity for the
inner-outer algorithm of the form 
$
I_1(k) \left[k^2 I_2(k) + \cO(k^3) \right] \ .   
$

In Table \ref{tab:IOImpHRAS}, we   illustrate the performance of this inner-outer method (which we denote {\tt IO-ImpHRAS}) 
implemented within the FGMRES 
framework. The outer tolerance is $10^{-6}$ (as before, and the inner tolerance $\tau$ 
is as indicated). Below each iteration count we present also the total running time of our reference NumPy implementation; this includes the setup time of all (sub)matrices. We also give an average time for each outer iteration. 

We see from the table that the commonly-used choice $\eps=k^2$ performs much worse that the choice of $\eps=k^\beta$ for the values of $\beta<2$ considered here; of these latter choices, $\eps=k$ seems to be the best choice for this composite algorithm.
The best times were obtained for a fairly large inner tolerance $\tau$; the 
results for $\eps =k, \tau = 0.5$ show a growth of the total compute time with $k$ that is close to $\mathcal{O}(k^4) = \mathcal{O}(n^{4/3})$.  Note that the runs were performed on a single CPU core; the method is highly parallelisable and  parallel implementation results will be presented in a future paper.      
\vspace{0.2cm}

\begin{table}
\begin{center}
\begin{tiny}
\begin{tabular}{|c||c|c|c|c|c|c|c|}
\hline 
\multicolumn{8}{|c|}{$\tau=0.01$}\tabularnewline
\hline 
$k\backslash\beta$  & \textbf{0 } & \textbf{0.4 } & \textbf{0.8 } & \textbf{1 } & \textbf{1.2 } & \textbf{1.6 } & \textbf{2.0}\tabularnewline
\hline 
\multirow{2}{*}{\textbf{10 }} & \textbf{15(5) } & \textbf{15(5) } & \textbf{15(5) } & \textbf{15(5) } & \textbf{15(5) } & \textbf{17(4) } & \textbf{21(3)}\tabularnewline
 & \textit{\emph{\scriptsize{0.69 {[}0.02{]}}}} & \textit{\emph{\scriptsize{   0.75 {[}0.02{]}}}} & \textit{\emph{\scriptsize{   0.73 {[}0.03{]}}}} & \textit{\emph{\scriptsize{   0.77 {[}0.03{]}}}} & \textit{\emph{\scriptsize{   0.74 {[}0.03{]}}}} & \textit{\emph{\scriptsize{0.75 {[}0.02{]}}}} & \textit{\emph{\scriptsize{0.82 {[}0.02{]}}}}\tabularnewline
\multirow{2}{*}{\textbf{20 }} & \textbf{17(7) } & \textbf{17(7) } & \textbf{17(7) } & \textbf{17 (6) } & \textbf{18(6) } & \textbf{20(4) } & \textbf{34 (2) }\tabularnewline
 & \textit{\emph{\scriptsize{4.49 {[}0.12{]}}}} & \textit{\emph{\scriptsize{   4.38 {[}0.12{]}}}} & \textit{\emph{\scriptsize{   4.34 {[}0.11{]}}}} & \textit{\emph{\scriptsize{   4.35 {[}0.11{]}}}} & \textit{\emph{\scriptsize{   4.42 {[}0.11{]}}}} & \textit{\emph{\scriptsize{   4.38 {[}0.10{]}}}} & \textit{\emph{\scriptsize{   5.18 {[}0.08{]}}}}\tabularnewline
\multirow{2}{*}{\textbf{40 }} & \textbf{21(11) } & \textbf{21(11) } & \textbf{21(11) } & \textbf{21(10)} & \textbf{21(9) } & \textbf{26(6) } & \textbf{56 (2) }\tabularnewline
 & \textit{\emph{\scriptsize{62.9 {[}0.96{]}}}} & \textit{\emph{\scriptsize{  62.8 {[}0.96{]}}}} & \textit{\emph{\scriptsize{  63.1 {[}0.96{]}}}} & \textit{\emph{\scriptsize{  62.4 {[}0.93{]}}}} & \textit{\emph{\scriptsize{  62.1 {[}0.91{]}}}} & \textit{\emph{\scriptsize{  63.3 {[}0.81{]}}}} & \textit{\emph{\scriptsize{  82.8 {[}0.73{]}}}}\tabularnewline
\multirow{2}{*}{\textbf{60} } & \textbf{27(15) } & \textbf{27(15) } & \textbf{27(15) } & \textbf{27(14) } & \textbf{27(12) } & \textbf{33(6) } & \textbf{78(2)}\tabularnewline
 & \textit{\emph{\scriptsize{420.9 {[}4.13{]}}}} & \textit{\emph{\scriptsize{ 422 {[}4.13{]}}}} & \textit{\emph{\scriptsize{ 423 {[}4.10{]}}}} & \textit{\emph{\scriptsize{ 421 {[}4.01{]}}}} & \textit{\emph{\scriptsize{ 416 {[}3.87{]}}}} & \textit{\emph{\scriptsize{ 426 {[}3.46{]}}}} & \textit{\emph{\scriptsize{ 560 {[}3.21{]}}}}\tabularnewline
\multirow{2}{*}{\textbf{80 }} & \textbf{36(17) } & \textbf{36(17) } & \textbf{35(16) } & \textbf{35(15)} & \textbf{34(12) } & \textbf{40(6) } & \textbf{101(2)}\tabularnewline
 & \textit{\emph{\scriptsize{1536 {[}11.1{]}}}} & \textit{\emph{\scriptsize{1564 {[}11.1{]}}}} & \textit{\emph{\scriptsize{1608 {[}10.89{]}}}} & \textit{\emph{\scriptsize{1555 {[}10.5{]}}}} & \textit{\emph{\scriptsize{1540 {[}10.0{]}}}} & \textit{\emph{\scriptsize{1542 {[}8.91{]}}}} & \textit{\emph{\scriptsize{2052 {[}8.54{]}}}}\tabularnewline
\multirow{2}{*}{\textbf{100 }} & \textbf{47(21)} & \textbf{47(21) } & \textbf{46(20)} & \textbf{45(18)} & \textbf{43(15)} & \textbf{48(7) } & \textbf{123(2)}\tabularnewline
 & \textit{\emph{\scriptsize{4061 {[}20.9{]}}}} & \textit{\emph{\scriptsize{4281 {[}21.1{]}}}} & \textit{\emph{\scriptsize{4073 {[}19.6{]}}}} & \textit{\emph{\scriptsize{3992 {[}20.5{]}}}} & \textit{\emph{\scriptsize{4078 {[}18.4{]}}}} & \textit{\emph{\scriptsize{3880 {[}17.1{]}}}} & \textit{\emph{\scriptsize{5130 {[}17.5{]}}}}\tabularnewline
\hline 
\multicolumn{8}{|c|}{$\tau=0.1$}\tabularnewline
\hline 
$k\backslash\beta$  & \textbf{0 } & \textbf{0.4 } & \textbf{0.8 } & \textbf{1 } & \textbf{1.2 } & \textbf{1.6 } & \textbf{2.0}\tabularnewline
\hline 
\multirow{2}{*}{\textbf{10} } & \textbf{15(3) } & \textbf{15(3) } & \textbf{15(3) } & \textbf{15(3) } & \textbf{15(3) } & \textbf{17(2) } & \textbf{19(2)}\tabularnewline
 & {\scriptsize{0.60~{[}0.02{]}}} & {\scriptsize{   0.60~{[}0.02{]}}} & {\scriptsize{ 0.61~{[}0.02{]}}} & {\scriptsize{   0.60~{[}0.02{]}}} & {\scriptsize{ 0.59~{[}0.02{]}}} & {\scriptsize{   0.60~{[}0.02{]}}} & {\scriptsize{ 0.73~{[}0.02{]}}}\tabularnewline
\multirow{2}{*}{\textbf{20} } & \textbf{17(4) } & \textbf{17(4) } & \textbf{17(4) } & \textbf{17(4) } & \textbf{18(4) } & \textbf{20(3) } & \textbf{34(1) }\tabularnewline
 & {\scriptsize{3.88~{[}0.10{]}}} & {\scriptsize{   3.84~{[}0.09{]}}} & {\scriptsize{   3.78~{[}0.09{]}}} & {\scriptsize{   3.91~{[}0.09{]}}} & {\scriptsize{   3.93~{[}0.09{]}}} & {\scriptsize{   3.98~{[}0.09{]}}} & {\scriptsize{   4.86~{[}0.08{]}}}\tabularnewline
\multirow{2}{*}{\textbf{40} } & \textbf{21(7) } & \textbf{21(7) } & \textbf{21(7) } & \textbf{21(6)} & \textbf{21(6) } & \textbf{26(4) } & \textbf{56(1) }\tabularnewline
 & {\scriptsize{ 56.9~{[}0.83{]}}} & {\scriptsize{56.9~{[}0.83{]}}} & {\scriptsize{  56.8~{[}0.82{]}}} & {\scriptsize{56.4~{[}0.81{]}}} & {\scriptsize{  56.8~{[}0.83{]}}} & {\scriptsize{  60.3~{[}0.76{]}}} & {\scriptsize{  80.6~{[}0.70{]}}}\tabularnewline
\multirow{2}{*}{\textbf{60} } & \textbf{27(9)} & \textbf{ 27(9)} & \textbf{ 27(8)} & \textbf{ 27(8)} & \textbf{ 27(7)} & \textbf{ 34(4)} & \textbf{ 82(1)}\tabularnewline
 & {\scriptsize{386~{[}3.65{]}}} & {\scriptsize{ 384~{[}3.63{]}}} & {\scriptsize{ 384~{[}3.58{]}}} & {\scriptsize{ 382~{[}3.55{]}}} & {\scriptsize{379~{[}3.46{]}}} & {\scriptsize{396~{[}3.24{]}}} & {\scriptsize{ 543~{[}3.13{]}}}\tabularnewline
\multirow{2}{*}{\textbf{80 }} & \textbf{36(10)} & \textbf{ 35(10)} & \textbf{ 35(10)} & \textbf{ 35(9)} & \textbf{ 34(8)} & \textbf{ 40(4)} & \textbf{104(1)}\tabularnewline
 & {\scriptsize{1361~{[}9.55{]}}} & {\scriptsize{1389~{[}9.51{]}}} & {\scriptsize{1398~{[}9.40{]}}} & {\scriptsize{1344~{[}9.34{]}}} & {\scriptsize{1368~{[}9.10{]}}} & {\scriptsize{1332~{[}8.49{]}}} & {\scriptsize{1926~{[}8.35{]}}}\tabularnewline
\multirow{2}{*}{\textbf{100 }} & \textbf{47(13)} & \textbf{ 46(13)} & \textbf{ 46(12)} & \textbf{ 45(11)} & \textbf{ 43(10)} & \textbf{ 48(4)} & \textbf{126(1)}\tabularnewline
 & {\scriptsize{3699~{[}18.6{]}}} & {\scriptsize{3723~{[}19.3{]}}} & {\scriptsize{3700~{[}18.1{]}}} & {\scriptsize{3535~{[}17.1{]}}} & {\scriptsize{3687~{[}16.7{]}}} & {\scriptsize{3530~{[}16.1{]}}} & {\scriptsize{4935~{[}16.9{]}}}\tabularnewline
\hline 
\multicolumn{8}{|c|}{$\tau=0.5$}\tabularnewline
\hline 
$k\backslash\beta$  & \textbf{0 } & \textbf{0.4 } & \textbf{0.8 } & \textbf{1 } & \textbf{1.2 } & \textbf{1.6 } & \textbf{2.0}\tabularnewline
\hline 
\multirow{2}{*}{\textbf{10} } & \textbf{17(2) } & \textbf{17(1) } & \textbf{18(1) } & \textbf{18(1) } & \textbf{18(1) } & \textbf{19(1) } & \textbf{23(1)}\tabularnewline
 & {\scriptsize{0.61~{[}0.02{]}}} & {\scriptsize{ 0.59~{[}0.02{]}}} & {\scriptsize{ 0.66~{[}0.01{]}}} & {\scriptsize{ 0.66~{[}0.02{]}}} & {\scriptsize{ 0.65~{[}0.02{]}}} & {\scriptsize{ 0.66~{[}0.01{]}}} & {\scriptsize{ 0.65~{[}0.01{]}}}\tabularnewline
\multirow{2}{*}{\textbf{20 }} & \textbf{19(2) } & \textbf{19(2) } & \textbf{19(2) } & \textbf{19(2) } & \textbf{19(2) } & \textbf{25(1) } & \textbf{36(1)}\tabularnewline
 & {\scriptsize{3.86~{[}0.08{]}}} & {\scriptsize{   3.72~{[}0.08{]}}} & {\scriptsize{3.72~{[}0.08{]}}} & {\scriptsize{3.68~{[}0.08{]}}} & {\scriptsize{   3.66~{[}0.08{]}}} & {\scriptsize{   4.00~{[}0.07{]}}} & {\scriptsize{   4.96~{[}0.07{]}}}\tabularnewline
\multirow{2}{*}{\textbf{40 }} & \textbf{22(4) } & \textbf{22(4) } & \textbf{22(4) } & \textbf{22(3) } & \textbf{22(3) } & \textbf{28(2) } & \textbf{61(1) }\tabularnewline
 & {\scriptsize{54.8~{[}0.73{]}}} & {\scriptsize{54.9~{[}0.73{]}}} & {\scriptsize{  54.8~{[}0.72{]}}} & {\scriptsize{  54.7~{[}0.71{]}}} & {\scriptsize{  54.8~{[}0.71{]}}} & {\scriptsize{  58.0~{[}0.69{]}}} & {\scriptsize{  80.4~{[}0.68{]}}}\tabularnewline
\multirow{2}{*}{\textbf{60 }} & \textbf{28(5)} & \textbf{ 28(5)} & \textbf{ 28(5)} & \textbf{ 28(5)} & \textbf{ 28(4)} & \textbf{ 35(2)} & \textbf{ 82(1)}\tabularnewline
 & {\scriptsize{370~{[}3.20{]}}} & {\scriptsize{ 371~{[}3.20{]}}} & {\scriptsize{372~{[}3.19{]}}} & {\scriptsize{ 370~{[}3.16{]}}} & {\scriptsize{ 369~{[}3.11{]}}} & {\scriptsize{ 383~{[}3.00{]}}} & {\scriptsize{539 {[}3.10{]}}}\tabularnewline
\multirow{2}{*}{\textbf{80 }} & \textbf{36(6)} & \textbf{ 36(6)} & \textbf{ 36(6)} & \textbf{ 36(5)} & \textbf{ 35(5)} & \textbf{ 42(2)} & \textbf{104(1)}\tabularnewline
 & {\scriptsize{1288~{[}8.62{]}}} & {\scriptsize{1375~{[}8.69{]}}} & {\scriptsize{1300~{[}8.59{]}}} & {\scriptsize{1316~{[}8.51{]}}} & {\scriptsize{1273~{[}8.38{]}}} & {\scriptsize{1323~{[}8.08{]}}} & {\scriptsize{1909~{[}8.19{]}}}\tabularnewline
\multirow{2}{*}{\textbf{100 }} & \textbf{46(8)} & \textbf{ 46(8)} & \textbf{ 46(7)} & \textbf{ 45(7)} & \textbf{ 44(6)} & \textbf{ 49(2)} & \textbf{126(1)}\tabularnewline
 & {\scriptsize{3533~{[}16.5{]}}} & {\scriptsize{3678~{[}16.01{]}}} & {\scriptsize{3586~{[}16.4{]}}} & {\scriptsize{3471~{[}15.9{]}}} & {\scriptsize{3483~{[}16.2{]}}} & {\scriptsize{3503~{[}15.5{]}}} & {\scriptsize{4832~{[}16.4{]}}}\tabularnewline
\hline 
\end{tabular}
\end{tiny}
\end{center} 
\caption{${\tt IO-ImpHRAS}$ with     
  $\epsprob = 0$ and $\epsprec = k^{\beta}$. Bold font: number of outer (inner) iterations). \red{Non-bold} font: total time in seconds [with an average time for each outer iteration in square brackets]  
} 
\label{tab:IOImpHRAS}
\end{table}

\noindent 
{\bf Acknowledgements.} \  We thank Melina Freitag, 
Stefan G\"{u}ttel, Jennifer Pestana and Andy Wathen for valuable advice 
about various aspects of weighted GMRES.  \red{We also thank Clemens Pechstein for valuable comments.}

\end{document}